\documentclass[11pt]{amsart}
\usepackage[colorlinks=true, pdfstartview=FitV, linkcolor=blue, 
citecolor=blue]{hyperref}

\usepackage{amssymb,bbm,amscd}
\usepackage{graphicx}
\usepackage{a4wide}

\theoremstyle{plain}
\newtheorem{theorem}{Theorem}[section]
\newtheorem{prop}[theorem]{Proposition}
\newtheorem{lemma}[theorem]{Lemma}
\newtheorem{coro}[theorem]{Corollary}
\newtheorem{fact}[theorem]{Fact}

\theoremstyle{definition}
\newtheorem{remark}[theorem]{Remark}

\numberwithin{equation}{section}
 
\newcommand{\dd}{\,\mathrm{d}}
\newcommand{\ii}{\ts\mathrm{i}}
\newcommand{\ee}{\,\mathrm{e}}
\newcommand{\ts}{\hspace{0.5pt}}
\newcommand{\nts}{\hspace{-0.5pt}}

\DeclareMathOperator{\dens}{\mathrm{dens}}
\DeclareMathOperator{\card}{\mathrm{card}}
\DeclareMathOperator{\tri}{|\nts\nts |\nts\nts |}

\newcommand{\vL}{\varLambda}
\newcommand{\vU}{\varUpsilon}
\newcommand{\cA}{\mathcal{A}}
\newcommand{\cB}{\mathcal{B}}
\newcommand{\cD}{\mathcal{D}}
\newcommand{\cE}{\mathcal{E}}
\newcommand{\cI}{\mathcal{I}}
\newcommand{\cM}{\mathcal{M}}
\newcommand{\cO}{\mathcal{O}}
\newcommand{\ZZ}{\mathbb{Z}{\ts}}
\newcommand{\RR}{\mathbb{R}\ts}
\newcommand{\CC}{\mathbb{C}\ts}
\newcommand{\EE}{\mathbb{E}}
\newcommand{\MM}{\mathbb{M}}
\newcommand{\NN}{\mathbb{N}}

\newcommand{\XX}{\mathbb{X}}
\newcommand{\YY}{\mathbb{Y}}
\newcommand{\one}{\mathbbm{1}}
\newcommand{\exend}{\hfill $\Diamond$}

\newcommand{\myfrac}[2]{\frac{\raisebox{-2pt}{$#1$}}
      {\raisebox{0.5pt}{$#2$}}}

\begin{document}

\title[Diffraction of a non-Pisot inflation]{Geometric 
  properties of a binary non-Pisot inflation\\[2mm]
  and absence of absolutely continuous diffraction}

\author{Michael Baake} 
\address{Fakult\"{a}t f\"{u}r Mathematik,
  Universit\"{a}t Bielefeld,\newline \hspace*{\parindent}Postfach
  100131, 33501 Bielefeld, Germany}
\email{mbaake@math.uni-bielefeld.de}

\author{Natalie Priebe Frank}
\address{Department of Mathematics and Statistics,
Vassar College, \newline
\hspace*{\parindent}Poughkeepsie, NY 12604, USA}
\email{nafrank@vassar.edu}

\author{Uwe Grimm} 
\address{School of Mathematics and Statistics,
  The Open University,\newline \hspace*{\parindent}Walton Hall, 
  Milton Keynes MK7 6AA, United Kingdom} 
\email{uwe.grimm@open.ac.uk}

\author{E.\ Arthur Robinson, Jr.}
\address{Department of Mathematics, 
George Washington University, \newline
\hspace*{\parindent}Washington, DC 20052, USA}
\email{robinson@gwu.edu}

\begin{abstract}
  One of the simplest non-Pisot substitution rules is investigated in
  its geometric version as a tiling with intervals of natural length
  as prototiles. Via a detailed renormalisation analysis of the pair
  correlation functions, we show that the diffraction measure cannot
  comprise any absolutely continuous component. This implies that the
  diffraction, apart from a trivial Bragg peak at the origin, is
  purely singular continuous. En route, we derive various geometric
  and algebraic properties of the underlying Delone dynamical system,
  which we expect to be relevant in other such systems as well.
\end{abstract}

\maketitle
\thispagestyle{empty}

\section{Introduction}

The spectral structure of substitution dynamical systems is well
studied, and many results are known; see \cite{Q} for a systematic
introduction. The theory is in good shape for substitutions of
constant length, both in one and in higher dimensions; see
\cite{Dekking,Rob,Nat1,squiral,Bartlett} as well as \cite{TAO} and
references therein. This is due to the fact that, for these systems,
the symbolic side and the geometric realisation with tiles of natural
size coincide, which also leads to a rather direct relation between
the diffraction measures of the system (and its factors) on the one
hand and the spectral measures on the other; see \cite{BLvE} and
references therein.

In general, the spectral theory of a substitution system and that of
its geometric counterpart can differ considerably \cite{CS}, in
particular when the inflation multiplier fails to be a
Pisot-Vijayaraghavan (PV) number \cite[Def.~2.13]{TAO}.  In fact,
beyond the substitutions of constant length, it often is simpler and
ultimately more revealing to use the geometric setting with natural
tile (or interval) sizes, as suggested by Perron--Frobenius theory. We
adopt this point of view below, and then speak of \emph{inflation
  rules} to make the distinction. Our entire analysis in this paper
will be in one dimension, where the tiles are just intervals.

Since rather little is known when one leaves the realm of PV inflation
multipliers, we present a detailed analysis of one of the simplest
non-Pisot (or non-PV) inflation rules on two letters, for which we
finally establish that the diffraction spectrum of the corresponding
Delone sets on the real line, apart from the trivial peak at $0$, is
purely singular continuous. En route, we shall encounter a number of
concepts and results that are described in some detail, in a way that
will facilitate generalisations to other inflation rules (and possibly
also to higher dimensions) in the future. A key ingredient to our
analysis is the study of the \emph{pair correlation functions} via
their exact renormalisation relations. The latter are analogous to
those recently derived \cite{BG15} for the Fibonacci inflation, where
they led to a spectral purity result and then to pure point
spectrum. This re-proved a known result in an independent way.

In the binary non-Pisot system studied below, the situation is more
complex because the spectrum is mixed, whence it remains to determine
the nature of the continuous part. To the best of our knowledge, the
answer is not in the literature, though the absence of absolutely
continuous components is certainly expected
\cite{Nat1,squiral,BG15,BerSol}. In anticipation of future work, we do
not present the shortest path to the result, as that would mean to
restrict more than necessary to methods that are limited to binary
alphabets and to this particular example. Instead, we use the concrete
system to investigate various concepts from \cite{BG15} in this more
complex case, with an eye to possible extensions and
generalisations. \smallskip

The paper is organised as follows. In Section~\ref{sec:prelim}, we
introduce the binary system via its symbolic substitution rule and the
matching geometric inflation tiling of the real line by two types of
intervals, following the general notions and results from
\cite{TAO}. Such a tiling is simultaneously considered as a
two-component Delone set, by taking the left endpoints of the
intervals as reference points. We also recall the construction of the
hull and its dynamical system, together with the key properties of the
latter. Section~\ref{sec:reno-coeff} introduces the pair correlation
functions and derives exact renormalisation relations, which are
then studied for their general solutions.  This part is not strictly
needed for our later analysis, but is interesting in its own right and
helps to understand the differences to the cases treated in
\cite{BG15}.

To continue, we need a reformulation of the pair correlation functions
in terms of translation bounded \emph{measures} and their Fourier
transforms, which is provided in Section~\ref{sec:reno-measure}. This
step emphasises the importance of two specific matrix families, whose
structure will later provide some arguments needed in the exclusion of
absolutely continuous diffraction. Section~\ref{sec:algebra} analyses
several properties of these matrix families by means of the (complex
resp.\ real) algebras generated by them. Once again, some of these
results go beyond what we need for our final goal, but highlight the
\emph{algebraic} structure of the problem.

Section~\ref{sec:explode} returns to the correlation measures and
their Fourier transforms. After splitting the transformed pair
correlation measures into their spectral parts (Lebesgue
decomposition), we rule out the existence of an absolutely continuous
component by a suitable iterated application of the renormalisation
relations in two directions. This approach employs the determination
of the corresponding extremal Lyapunov exponents, 
some details of which are given in
Appendix A.  Two underlying renormalisation arguments are further
explained in Appendix B, in the simpler setting of a scalar
equation. Section~\ref{sec:appl} covers an application to the
diffraction in the balanced weight case, where the pure point part is
extinct. In particular, we illustrate one specific case of a singular
continuous measure in this setting, based on a precise numerical
calculation of the corresponding (continuous) distribution function.

\section{Setting and preliminaries}\label{sec:prelim}

\subsection{Substitution, inflation and hull}
We consider the primitive two-letter substitution
\begin{equation}\label{subst}
  \varrho\! :\quad \begin{array}{l} 0 \mapsto 0111 \\ 
             1\mapsto 0\end{array}
\end{equation}
on the alphabet $\{0,1\}$. It defines a unique (symbolic) \emph{hull}
$\XX$, for instance via the shift orbit closure of the bi-infinite
fixed point $w$ of $\varrho^{2}$ with legal seed $w^{(0)}=0|0$,
\[
  0|0 \,\stackrel{\varrho^{2}}{\longmapsto}\,  w^{(1)}=0111000|0111000 
   \,\stackrel{\varrho^{2}}{\longmapsto}\,  \dots
   \,\stackrel{\varrho^{2}}{\longmapsto}\, w^{(i)}\; 
   \xrightarrow{\;i\to\infty\;}\; w=\varrho^{2}(w) \ts ,
\]
where finite words are considered as embedded into $\{0,1\}^{\ZZ}$ and
$|$ marks the reference point (between position $-1$ and $0$).  In
particular, $(\XX,\ZZ)$ with the continuous $\ZZ$-action generated by
the shift is a minimal topological dynamical system. There is
precisely one shift invariant probability measure $\mu^{}_{\XX}$ on
$\XX$, namely the patch (or word) frequency measure, so that
$(\XX,\ZZ,\mu^{}_{\XX})$ is strictly ergodic. The invariant measure is
intimately connected with the substitution origin; see
\cite{Q,Sol,TAO} for background.

The corresponding integer substitution matrix is
\begin{equation}\label{eq:submat}
     M \; = \; \begin{pmatrix}1 & 1 \\ 3 & 0\end{pmatrix} , 
\end{equation}
with irreducible characteristic polynomial
$\det( M - x \one) = x^2 - x - 3$ and Perron--Frobenius (PF)
eigenvalue $\lambda=(1+\sqrt{13}\,)/2 \approx 2.302 \ts 776$. Its
algebraic conjugate, which is also the second eigenvalue of $M$, is
$\lambda' = (1-\sqrt{13}\,)/2 = 1-\lambda \approx -1.302 \ts 776$,
which lies outside the unit circle, wherefore $\lambda$ is \emph{not}
a PV number. Since $\lvert \det (M) \rvert = 3$ is prime, $M$ has no
root in $\mathrm{Mat} (2,\ZZ)$, hence $\varrho$ cannot have a
substitutional root.

The statistically normalised right eigenvector $v^{}_{\text{PF}}$ to
the eigenvalue $\lambda$ is
\begin{equation}\label{eq:lettfreq}
    v^{}_{\text{PF}} \; = \; (\nu^{}_{0}, \nu^{}_{1})^{t} \; = \;
    \bigl(\lambda^{-1}, 3\lambda^{-2}\bigr)^{t}\; = \;
    \myfrac{1}{3}\,\bigl(\lambda-1, 4-\lambda\bigr)^{t}
    \; \approx \; \bigl(0.434, 0.566\bigr)^{t} ,
\end{equation}
which determines the (relative) frequencies of the two letters,
$\nu^{}_{0}$ and $\nu^{}_{1}$.  For a consistent geometric realisation
as an \emph{inflation rule} on two intervals, we use interval lengths
according to the left PF eigenvector of $M$, which we choose as
$(\lambda,1)$; see Figure~\ref{fig:infl} for an illustration. Note
that this choice is particularly simple from an algebraic point of
view, as the two lengths are the generating elements of
$\ZZ[\lambda]$. From a dynamical perspective, it would perhaps be more
natural to make a choice with average length $1$, but this would give
a more complicated algebraic structure for the coordinates.

\begin{figure}
\begin{center}
  \includegraphics[width=0.84\textwidth]{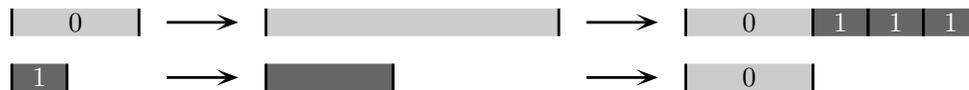}
\end{center}
\caption{\label{fig:infl} Inflation rule for the substitution
  $\varrho$ of Eq.~\eqref{subst}. It consists of two steps, an
  expansion of the intervals by $\lambda$ and the ensuing subdivision
  into intervals of the original length in the correct order. }
\end{figure}

Now, let $\vL^{w}$ denote the point set of left interval end points
that corresponds to our above fixed point $w$, so
\begin{equation}\label{eq:lam-set}
    \vL^{w} \; =\; \bigl\{\ldots, -1\!-\!3\lambda, 
     -3\lambda, -2\lambda,
       -\lambda, 0, \lambda, 1\!+\!\lambda, 
     2\!+\!\lambda, 3\!+\!\lambda,
       3\!+\!2\lambda,\ldots\bigr\} \; \subset\; \ZZ[\lambda]\, .
\end{equation}
This is a Delone set of density
$\dens (\vL^{w})=(6+\lambda)/13\approx 0.638{\ts}675$, and its orbit
closure under the natural translation action of $\RR$ defines the
\emph{geometric hull}
\begin{equation}\label{eq:def-cont-hull}
    \YY \, := \, \overline{\{ t + 
     \varLambda_{\phantom{\hat{\hat{I}}}}^{w} :
      t\in\RR \}}^{\mathrm{\ts LT}} \ts ,
\end{equation}
where the closure is taken in the local topology; see \cite{TAO} for
details. Note that $\YY$ is compact in the local topology, as a result
of $\varLambda^{w}$ being a Delone set of \emph{finite local
  complexity}, which means that the Minkowski difference
$\varLambda^{w}-\varLambda^{w}$ is a locally finite set.

Now, $(\YY,\RR)$ is once again a minimal topological dynamical system,
which still only has one invariant probability measure, namely the one
induced by $\mu^{}_{\XX}$ from above, so that $(\YY,\RR,\mu^{}_{\YY})$
is strictly ergodic, too. This system can be obtained as a suspension
of the previous one, with a non-constant roof function; see
\cite[Ch.~11]{CFS} for background. It is this latter system that we
investigate now in more detail. To this end, we also need the
Minkowski difference
\[
   \varDelta \, := \,  \vL^{w}-\vL^{w} \, = \, 
   \bigl\{0, \pm 1, \pm 2, \pm \lambda, 
    \pm 3, \pm (1\! +\! \lambda), \pm (2\! +\! \lambda), \pm 2\lambda, 
   \pm (3\! +\! \lambda), \pm (1\! +\! 2\lambda),
    \ldots \bigr\} ,
\]
which is the set of distances between points in $\vL^{w}$ and
satisfies $\varDelta \subset \ZZ[\lambda]$.

\begin{prop}\label{prop:hull}
  Let\/ $\YY$ be the geometric hull from Eq.~\eqref{eq:def-cont-hull},
  and\/ $\vL\in\YY$.  Then, the Minkowski difference\/
  $\varDelta_{\vL} := \vL -\vL$ is a locally finite subset of\/ $\RR$,
  but it is not uniformly discrete. One has\/ $\varDelta_{\vL} =
  \varDelta$ for all\/ $\vL\in\YY$, so the difference set is constant
  on\/ $\YY$.
\end{prop}

\begin{proof}
  The hull $\YY$ is the translation orbit closure of the Delone set
  $\vL^{w}$ from Eq.~\eqref{eq:lam-set}.  Since $\lambda
  \vL^{w}\subset \vL^{w}$ by construction, but $\lambda$ is not a PV
  number, the set $\vL^{w}$ cannot be a Meyer set; compare
  \cite[Thm.~2.4]{TAO}. Consequently, $\vL^{w} - \vL^{w}$ is discrete,
  but not uniformly discrete.\footnote{The failure of uniform
    discreteness comes from a property of the sequence
    $(\lambda^{m})^{}_{m\in\NN}$ which ultimately results in
    distances between neighbouring points of $\vL-\vL$ not being
    bounded from below.} Since the difference set is also closed, it
  is locally finite.

  The inflation rule derived from $\varrho$ is primitive, whence we
  know form standard arguments (compare \cite[Ch.~4]{TAO} for details)
  that the hull $\YY$ consists of the LI class of $\vL^{w}$, which
  means that any two elements of $\YY$ are locally
  indistinguishable. This implies that $\varDelta_{\vL}$ with
  $\vL\in\YY$ cannot depend on $\vL$, which establishes the second
  claim.
\end{proof}

In what follows, we will freely move between the tiling picture and
its representation as a Delone set, where we tacitly make use of the
equivalence concept of \emph{mutual local derivability} (MLD); compare
\cite[Sec.~5.2]{TAO} and references therein for background.

\subsection{Natural autocorrelation and diffraction}
Next, let us recall the notion of the natural autocorrelation of
$\vL\in\YY$, compare \cite{Hof} or \cite[Def.~9.1]{TAO}, which is
usually done by first turning $\vL$ into the Dirac comb
$\delta^{}_{\!\vL} := \sum_{x\in\vL} \delta_{x}$. The latter is both a
tempered distribution and a translation bounded measure on $\RR$. In
particular, $\delta^{}_{\!\vL}$ is \emph{not} a finite measure. Its
\emph{autocorrelation} $\gamma$ is then defined as the volume-averaged
(or Eberlein) convolution
\[
    \gamma \, = \, \delta^{}_{\!\vL} \circledast
    \delta^{}_{-\nts\vL} \, := \lim_{n\to\infty}
    \frac{\delta^{}_{\!\vL_{r}} \! * \delta^{}_{-\nts\vL_{r}}}{2 \ts r}
\]
where $\vL_{r} = \vL \cap [-r,r]$. The existence of the limit, for any
$\vL\in\YY$, is a consequence of unique ergodicity. Once again,
$\gamma$ is not a finite measure, but it is translation bounded.  In
fact, from a simple calculation together with
Proposition~\ref{prop:hull}, one can see that
$\gamma = \sum_{z\in \varDelta} \eta (z) \ts \delta_{z}$ with the
autocorrelation coefficients
\[
    \eta (z) \, = \lim_{r\to\infty} 
    \frac{\card \bigl(\vL_{r} \cap (z + \vL_{r})\bigr)}{2 \ts r}
    \, =  \lim_{r\to\infty} 
    \frac{\card \bigl(\vL_{r} \cap (z + \vL)\bigr)}{2 \ts r} .
\]
More generally, we also need to consider \emph{weighted} Dirac combs
$\omega$ with (generally complex) weights $u^{}_{0}$ and $u^{}_{1}$
for the two letters (or point types). They are defined as
\begin{equation}\label{eq:comb}
  \omega \, = \, \omega^{}_{\!\varLambda,u} \, = 
   \sum_{x\in\vL} u(x)\, \delta_{x} 
\end{equation}
with $u(x)\in\{u^{}_{0}, u^{}_{1}\}$ depending on whether $x$ is the
left endpoint of an interval of type $0$ or $1$. In other words, we
consider $\vL = \vL^{(0)}\ts \dot{\cup} \vL^{(1)}$ as the disjoint
union of two Delone sets, according to the two types of points in
$\vL$, and hence as a two-component Delone set. In our case at hand,
the two versions are MLD, because the prototiles (intervals) have
different lengths.  Note that, for $i\in\{ 0,1\}$, we then have the
density relations
\begin{equation}\label{eq:dens-rel}
    \dens \bigl(\vL^{(i)} \bigr) \, = \, \nu^{}_{i} \, \dens (\vL) 
\end{equation}
with the frequencies $\nu^{}_{i}$ from Eq.~\eqref{eq:lettfreq}.

The natural autocorrelation of such a Dirac comb $\omega$ is 
\begin{equation}\label{eq:Eberlein}
   \gamma^{}_{u} \, = \,
   \omega \circledast \widetilde{\omega}
         \, :=  \lim_{r\to\infty} \frac{\omega_{r}\nts *  
        \widetilde{\omega_{r}}}{2 \ts r} \ts ,
\end{equation}
with $\omega = \omega^{}_{\!\varLambda,u}$ and $\omega_{r}$ denoting
the restriction of $\omega$ to the interval $[-r,r]$. Moreover, the
twisted measure $\widetilde{\mu}$ is defined by
$\widetilde{\mu} (g) = \overline{\mu (\widetilde{g})}$ for test
functions $g\in C_{\mathsf{c}} (\RR)$, with
$\widetilde{g} (x) := \overline{g(-x)}$; see \cite[Ch.~9.1]{TAO} for
background. As before, the existence of the limit is a consequence of
unique ergodicity, and we have
$ \gamma^{}_{u} = \sum_{z\in\varDelta} \eta^{}_{u} (z)\ts
\delta^{}_{z}$, this time with
\begin{equation}\label{eq:gen-eta}
   \eta^{}_{u} (z) \; = \, \lim_{r\to\infty} \,\myfrac{1}{2 \ts r} \!
   \sum_{y,y+z\in \vL_{r}}\!\! \overline{u(y)}\, u(y+z) \ts .
\end{equation}
By construction, any such autocorrelation $\gamma^{}_{u}$ is a
\emph{positive definite} measure, which means that
$\gamma^{}_{u} (g * \widetilde{g}) \geqslant 0$ holds for all
$g\in C_{\mathsf{c}} (\RR)$. This is significant because any positive
definite measure is Fourier transformable as a measure, which is a
non-trivial statement since $\gamma^{}_{u}$ is not a finite
measure. Its Fourier transform is then a positive measure, as a
consequence of the Bochner--Schwartz theorem; see \cite{BF,RS} for
background.

\begin{prop}\label{prop:hull-2}
  Given arbitrary weights\/ $u^{}_{0}, u^{}_{1} \in \CC$ for the two
  types of points, the autocorrelation measure\/ $\gamma^{}_{u}$ is
  positive definite, and it is the same for all\/ $\vL\in\YY$, which
  means that\/ $\gamma^{}_{u}$ is the autocorrelation both for an
  arbitrary element of the hull and for the entire hull. The analogous
  statement holds for the diffraction measure\/
  $\widehat{\gamma^{}_{u}}$, which is always a translation bounded,
  positive measure.
\end{prop}

\begin{proof}
  As explained above, the first claim is a consequence of
  Eq.~\eqref{eq:Eberlein}, because positive definiteness of measures is
  preserved under limits in the vague topology.

  The second claim on the autocorrelation follows from the uniform
  existence of patch frequencies together with the fact that any
  two elements of the hull are locally indistinguishable.

  The statement on the diffraction then is a consequence of the
  uniqueness of the Fourier transform. The positivity of
  $\widehat{\gamma^{}_{u}}$ is clear by Bochner--Schwartz, while its
  translation boundedness follows from \cite[Prop.~4.9]{BF}.
\end{proof}

The Fourier transform $\widehat{\gamma^{}_{u}}$ is called the
\emph{diffraction measure} of $\omega$, which is thus always a
positive measure. With respect to Lebesgue measure on $\RR$, it has
the unique decomposition
\[
   \widehat{\gamma^{}_{u}} \; = \; 
   \bigl(\widehat{\gamma^{}_{u}} \bigr)_{\mathsf{pp}} +
   \bigl(\widehat{\gamma^{}_{u}} \bigr)_{\mathsf{sc}} +
   \bigl(\widehat{\gamma^{}_{u}} \bigr)_{\mathsf{ac}}
\]
into its pure point, singular continuous and absolutely continuous
parts; see \cite[Rem.~9.3]{TAO} for more. Our aim is to determine the
precise nature of $\widehat{\gamma}$ for our system.

Because we are dealing with a non-PV inflation multiplier
$\lambda$, we get the following result on the pure point part of the
diffraction measure.

\begin{theorem}\label{thm:pp}
  Let\/ $\vL = \vL^{(0)}\ts \dot{\cup} \vL^{(1)}$ be a fixed element
  of the hull\/ $\YY$.  Consider the weighted Dirac comb\/
  $\omega^{}_{\!  \vL,u} = u^{}_{0} \,
  \delta^{}_{\!\!\vL_{\vphantom{a}}^{(0)}} \nts + u^{}_{1}\,
  \delta^{}_{\!\!\vL_{\vphantom{a}}^{(1)}}$,
  with arbitrary complex weights\/ $u^{}_{0}$ and\/ $u^{}_{1}$.  Then,
  the pure point part of the corresponding diffraction measure\/
  $\widehat{\gamma^{}_{\!\vL,u}}$ is given by
\[
  \bigl(\widehat{\gamma^{}_{u}}\bigr)_{\mathsf{pp}} 
  \; = \; I^{}_{0}\, \delta^{}_{0} \, ,
  \quad \text{with } \, I^{}_{0} \, = \, \big\lvert
  \nts\nts \dens (\vL^{w})\,
  (u \cdot v^{}_{\mathrm{PF}})\big\rvert^{2} \, = \, 
  \big\lvert \tfrac{2\lambda-1}{13}u^{}_{0} + 
  \tfrac{7-\lambda}{13} u^{}_{1} \big\rvert^{2}  .
\]
In particular, the pure point part is the same for all\/
$\vL \in \YY$.
\end{theorem}

\begin{proof}
  The claim is trivial for $u^{}_{0} = u^{}_{1} = 0$, so let us assume
  that at least one of the weights is nonzero. Then, the dynamical
  system $(\YY_{u}, \RR)$, as obtained by the vague closure of the set
  $\{ \delta_{t} * \omega^{}_{\! \vL, u} : t \in \RR \}$, is seen
  to be topologically conjugate to $(\YY, \RR)$ via standard 
  MLD arguments. We thus know from \cite[Thm.~4.3 and
  Cor.~4.5]{Sol} that we only have the trivial eigenfunction for
  $(\YY_{u}, \RR)$.

  Now, assume $\widehat{\gamma^{}_{u}} \bigl( \{ k \} \bigr) = I(k)
  \ne 0$ for some $k\ne 0$. We then know from \cite[Thm.~3.4]{Hof} and
  \cite[Thm.~5]{Lenz} that $I(k) = \big\lvert a^{}_{\nts \vL,u} (k)
  \big\rvert^{2}$, with the Fourier--Bohr coefficient
\[
   a^{}_{\nts \vL,u} (k) \, = \lim_{r\to\infty} \myfrac{1}{2 \ts r}
   \sum_{x \in \vL_{r}} u(x) \, \ee^{- 2 \pi \ii k x} 
\]
and $u(x)$ as in Eq.~\eqref{eq:comb}.  Note that this coefficient
exists for all $\vL\in\YY$ (and even uniformly so) for our system due
to unique ergodicity. Note also that the coefficient depends on $\vL$,
while $I(k)$ does not.  In fact, for fixed $u$, the map given by
$\vL \mapsto a^{}_{\nts \vL,u} (k)$ defines an eigenfunction (in fact,
a continuous one) of $(\YY_{u}, \RR)$ because
$a^{}_{t+\vL,u} (k) = \ee^{-2 \pi \ii kt} a^{}_{\nts\nts \vL,u} (k)$ for
$t\in\RR$; compare \cite{Hof,Lenz}. Since we know that such an
eigenfunction cannot exist, we must have $I(k) = 0$.

Finally, the formula for $I_{0}$ follows from an application of
\cite[Prop.~9.2]{TAO} by observing that, for a weighted Dirac comb
supported on $\vL$, the averaging formula stated there reduces to 
the absolute square of the volume-averaged weights, via 
Eq.~\eqref{eq:dens-rel}.
\end{proof}

In other words, we can only have the trivial central Bragg peak.
Later, we will be interested in the situation that $I^{}_{0} = 0$,
which we refer to as the \emph{balanced weight case}. Since no
recursive formula for the coefficients $\eta (z)$ is known, and since
it is desirable to have a systematic approach to $\gamma$ for
arbitrary choices of the weights, we now turn to the pair correlation
functions and their properties. This will produce another path to
$\eta$ and $\gamma$.

\section{Renormalisation approach to pair
correlation functions}\label{sec:reno-coeff}

As before, we use the tiling picture (with the two types of intervals
as prototiles) and the two-component Delone set picture in parallel.
These two versions obviously give topologically conjugate dynamical
systems, wherefore we simply identify them canonically. Which
representation we use will always be clear from the context.

\subsection{Pair correlation functions}
Let $\vL$ be any element of the (geometric) hull $\YY$, and let
$\nu^{}_{ij}(z)$ denote the relative frequency of distance $z$ from a
left endpoint of an interval of type $i$ to one of type $j$, with
$i,j \in\{0,1\}$.  Decomposing
$\vL = \vL^{(0)}\ts \dot{\cup} \vL^{(1)}$ as before, this means
\begin{equation}\label{eq:def-nuij}
    \nu^{}_{ij}(z) \; = \; \lim_{r\to\infty}
    \frac{\card \bigl(\vL^{(i)}_{r} \cap 
    (\vL^{(j)}_{r} - z)\bigr)}{\card (\vL_{r})} \; = \;
    \myfrac{1}{\dens (\vL)}\, \lim_{r\to\infty}
    \frac{\card \bigl(\vL^{(i)}_{r} \cap 
    (\vL^{(j)}_{r} - z)\bigr)}{2 \ts r} \ts .
\end{equation}
These limits exist for any $z\in\RR$, and one has $\nu^{}_{ij}(z)
\geqslant 0$.  The use of relative frequencies is advantageous because
they are dimensionless and thus simplify various statements below.

The four functions $\nu^{}_{ij}$, which we call the \emph{pair
  correlation functions}, are well-defined, as another consequence of
the unique ergodicity of our system. Clearly, for any
$i,j \in \{ 0,1 \}$ and any $z\in \RR$, they satisfy the symmetry
relations
\begin{equation}\label{eq:symm}
    \nu^{}_{ij}(z) \, = \, \nu^{}_{ji}(-z) \ts .
\end{equation}
Moreover, we have $\nu^{}_{ij}(z)=0$ for any $z\not\in\vL-\vL$. To
improve on this, for any $\vL \in \YY$, decompose $\vL = \vL^{(0)}\ts
\dot{\cup} \vL^{(1)}$ as above and define the point sets
\begin{equation}\label{eq:def-Sij}
    S_{ij} \, := \, \vL^{(j)} - \vL^{(i)} .
\end{equation}
By an obvious variant of Propositions~\ref{prop:hull} and
\ref{prop:hull-2}, it is clear that each $S_{ij}$ is again independent
of the choice of $\vL$, hence constant on the hull. Due to strict
ergodicity, we then have the following stronger property.

\begin{fact}\label{fact:support}
  The pair correlation functions\/ $\nu^{}_{ij}$ defined in
  Eq.~\eqref{eq:def-nuij} are independent of the choice of
  $\vL \in \YY$. They satisfy the symmetry relations of
  Eq.~\eqref{eq:symm}. Moreover, one has\/ $\nu^{}_{ij}(z) > 0$ if and
  only if\/ $z \in S_{ij}$, where\/ $S_{ij}$ is the set from
  Eq.~\eqref{eq:def-Sij}, which is the same for all\/ $\vL \in \YY$.
  \qed
\end{fact}

The autocorrelation coefficients $\eta^{}_{u}(z)$ from
Eq.~\eqref{eq:gen-eta} for the weighted Dirac comb $\omega$ of
Eq.~\eqref{eq:comb} can now be expressed in terms of the pair
correlation functions of $\YY$ as a quadratic form, namely as
\begin{equation}\label{eq:etanu}
   \eta^{}_{u}(z) \; = \; \mathrm{dens}(\vL) \!\! 
   \sum_{i,j \in\{0,1\}}\!\!
   \overline{u^{}_{i}} \; 
   \nu^{}_{ij}(z) \, u^{}_{j}  \ts . 
\end{equation}
This shows that the `natural' objects to study are indeed the pair
correlation functions, as their knowledge gives access to the
autocorrelation measures (and hence to their Fourier transforms) for
\emph{any} choice of the weights. Also, as mentioned earlier, we are
not aware of a functional relation that would determine the
coefficients $\eta^{}_{u}(z)$ directly. However, such a relation can
be derived from the inflation structure for the four pair correlation
functions $\nu^{}_{ij}(z)$; see \cite{BG15} for related examples.

\begin{prop}\label{prop:reno-eqs}
  The pair correlation functions\/ $\nu^{}_{ij}$ of the hull\/ $\YY$
  satisfy the following linear renormalisation equations,
\begin{align*}
\nu^{}_{00}(z) & \, = \, \myfrac{1}{\lambda} \,\Bigl(
  \nu^{}_{00}\bigl(\tfrac{z}{\lambda}\bigr) +
  \nu^{}_{01}\bigl(\tfrac{z}{\lambda}\bigr) +
  \nu^{}_{10}\bigl(\tfrac{z}{\lambda}\bigr) +
  \nu^{}_{11}\bigl(\tfrac{z}{\lambda}\bigr)\Bigr), \\
\nu^{}_{01}(z) & \, = \, \myfrac{1}{\lambda} \,\Bigl( 
  \nu^{}_{00}\bigl(\tfrac{z-\lambda}{\lambda}\bigr) +
  \nu^{}_{00}\bigl(\tfrac{z-1-\lambda}{\lambda}\bigr) +
  \nu^{}_{00}\bigl(\tfrac{z-2-\lambda}{\lambda}\bigr) + 
  \nu^{}_{10}\bigl(\tfrac{z-\lambda}{\lambda}\bigr) +
  \nu^{}_{10}\bigl(\tfrac{z-1-\lambda}{\lambda}\bigr) +
  \nu^{}_{10}\bigl(\tfrac{z-2-\lambda}{\lambda}\bigr)\Bigr), \\
\nu^{}_{10}(z) & \, = \, \myfrac{1}{\lambda} \,\Bigl(
  \nu^{}_{00}\bigl(\tfrac{z+\lambda}{\lambda}\bigr) +
  \nu^{}_{00}\bigl(\tfrac{z+1+\lambda}{\lambda}\bigr) +
  \nu^{}_{00}\bigl(\tfrac{z+2+\lambda}{\lambda}\bigr) +
  \nu^{}_{01}\bigl(\tfrac{z+\lambda}{\lambda}\bigr) +
  \nu^{}_{01}\bigl(\tfrac{z+1+\lambda}{\lambda}\bigr) +
  \nu^{}_{01}\bigl(\tfrac{z+2+\lambda}{\lambda}\bigr)\Bigr), \\
\nu^{}_{11}(z) & \, = \, \myfrac{1}{\lambda} \,\Bigl(
  3\, \nu^{}_{00}\bigl(\tfrac{z}{\lambda}\bigr) +
  2\, \nu^{}_{00}\bigl(\tfrac{z+1}{\lambda}\bigr) +
  2\, \nu^{}_{00}\bigl(\tfrac{z-1}{\lambda}\bigr) +
  \nu^{}_{00}\bigl(\tfrac{z+2}{\lambda}\bigr) +
  \nu^{}_{00}\bigl(\tfrac{z-2}{\lambda}\bigr) \Bigr) ,
\end{align*}
together with\/ $\nu^{}_{ij}(z)=0$ for any\/ $z\not\in S_{ij}$ and the
symmetry relations\/ $\nu^{}_{ji}(z)=\nu_{ij}(-z)$ for all\/ $z\in\RR$
and all\/ $i,j \in \{0,1\}$.
\end{prop}

\begin{figure}
\begin{center}
  \includegraphics[width=0.84\textwidth]{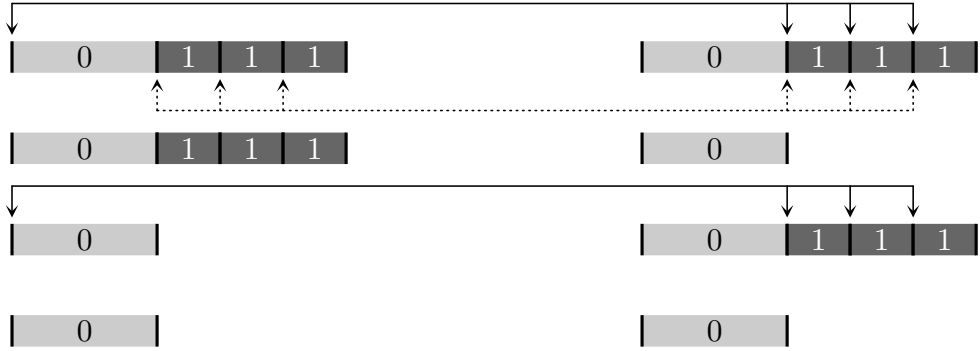}
\end{center}
\caption{\label{fig:recog} Illustration of the location of tiles
  within their level-$1$ supertiles for the proof of
  Proposition~\ref{prop:reno-eqs}. The solid arrows on top of the
  lines indicate the pairing for intervals of types $0$ (left) and $1$
  (right), the dashed arrows the pairings of two intervals of type
  $1$. The remaining two cases are analogous.}
\end{figure}

\begin{proof}
  Since our inflation rule is aperiodic, we have local recognisability
  \cite{Q}. This means that each tile in any (fixed) element of the
  hull lies inside a unique level-$1$ supertile that is identified by
  a local rule.  Concretely, each patch of type $0111$ constitutes a
  supertile of type $0$, while each tile of type $0$ that is followed
  by another $0$ (to the right) stands for a supertile of type
  $1$. Below, we simply say supertile, as no level higher than $1$
  will occur in this proof.

  Due to the inflation structure, it is also clear that the relative
  frequency (meaning relative to $\vL$) of two supertiles of type $i$
  and $j$ with distance $z$ (from $i$ to $j$) is given by
  $\frac{1} {\lambda}\ts \nu^{}_{ij}\bigl( \frac{z}{\lambda} \bigr)$.
  This follows from the simple observation that, for any $\vL\in\YY$,
  the point set of the left endpoints of the supertiles is a set of
  the form $\lambda \vL'$ for some $\vL'\in\YY$.

  We can now relate the occurrences of pairs of tiles at distance $z$
  to those of the supertiles they are in; see Figure~\ref{fig:recog}
  for an illustration. For instance, a distance $z$ between two tiles
  of type $0$ emerges once from any pair of supertiles (of either
  type) at the same distance.  With the above formula for the relative
  frequency of the supertiles, this gives the first equation.

  Likewise, the frequency $\nu^{}_{01}(z)$ is composed of supertile
  frequencies of type $00$, at distances $(z-\lambda)$,
  $(z-\lambda-1)$ and $(z-\lambda-2)$, and supertile frequencies of
  type $10$, at the same set of distances; see Figure~\ref{fig:recog}
  for an explicit illustration. This gives the second equation. The
  remaining two identities are derived analogously.

  The additional constraints are clear from Fact~\ref{fact:support}.
\end{proof}

\subsection{Solution space}
In view of our setting, it is clear that there is at least one
solution of the (infinite) linear system of equations in
Proposition~\ref{prop:reno-eqs}, under the extra conditions stated
there. Less obvious is the following, considerably stronger statement,
where a larger support is admitted and no symmetry relation is
prescribed.

\begin{theorem}\label{thm:reno-sols}
  Assume that $\nu^{}_{ij}(z) = 0$ for all\/ $z\not\in\varDelta$ and
  all\/ $i,j \in \{0,1\}$, and consider the subset of equations for\/
  $\nu^{}_{ij}(z)$ that emerges from
  Proposition~\emph{\ref{prop:reno-eqs}} by restricting to arguments\/
  $z\in\varDelta$ of modulus\/ $|z| \leqslant 1 + \lambda$. This is a finite
  and closed set of linear equations. The dimension of the solution
  space of all equations equals that of this finite subset.

  In particular, the dimension of the solution space is\/ $1$.
  Taking into account the requirement\/
  $\nu^{}_{00}(0)+\nu^{}_{11}(0)=1$ for the relative prototile
  frequencies, the solution is unique.
\end{theorem}

\begin{proof}
  Observe first that, when $\lvert z \rvert \leqslant \lambda + 1$, no
  argument on the right-hand sides of the identities in
  Proposition~\ref{prop:reno-eqs} exceeds $ \lambda + 1$ in
  modulus. Since
  $\{ z\in\varDelta : \lvert z \rvert \leqslant \lambda + 1\}$ is a
  finite set and $\nu^{}_{ij} (z)=0$ for any $z\not\in\varDelta$, the
  first claim is obvious. If, on the other hand,
  $\lvert z \rvert > \lambda + 1$, all arguments on the right-hand
  sides are strictly smaller than $\lvert z \rvert$, wherefore all
  coefficients $\nu^{}_{ij}(z)$ are determined by values at smaller
  arguments. Since $\varDelta$ is locally finite, the second claim
  follows from standard arguments.

  Thus, only the equations for
  $\lvert z \rvert \in \{0,1,2,\lambda, 3,1+\lambda\}$ need to be
  considered separately, as after that all frequencies are determined
  recursively. This gives $44$ linear equations that can be solved by
  standard methods, either by hand or by algebraic manipulation, the
  details of which we omit here.  It turns out that $\nu^{}_{00}(0)$
  is not fixed by the relations, while all other function values can
  be written as a function of $\nu^{}_{00} (0)$, with
  $\nu^{}_{11} (0) = \frac{3}{\lambda} \ts \nu^{}_{00} (0)$ in
  particular. This means that the solution space is indeed
  one-dimensional. Imposing $\nu^{}_{00}(0)+\nu^{}_{11}(0)=1$ with
  $\nu^{}_{00} (0)\geqslant 0$ results in
  $\nu^{}_{00}(0) = 1/\lambda$, so $\nu^{}_{00}(0) = \nu^{}_{0}$ and
  $\nu^{}_{11}(0) = \nu^{}_{1}$ with the frequencies from
  Eq.~\eqref{eq:lettfreq}. All other values are then uniquely
  determined as listed in Table~\ref{nutab}. We leave the details of
  this calculation to the reader.
\end{proof}

\begin{table}
\caption{Relative frequencies $\nu^{}_{ij}(z)$ for all
  distances $z\in\varDelta$ with $|z|\leqslant 1+\lambda$.\label{nutab}}
\renewcommand{\arraystretch}{1.5}
\begin{tabular}{|c|ccccccccccc|}
  \hline
 $z$ &  $-1-\lambda$ & $-3$ & $-\lambda$ & $-2$ & 
 $-1$ & $0$ &  $1$ & $2$ & $\lambda$ & $3$ & $1+\lambda$ \\ \hline
 $\nu^{}_{00}$ &  $0$ & $0$ & $3\lambda^{-3}$ &  $0$ &  
 $0$ &  $\lambda^{-1}$ & $0$ &  $0$ & 
 $3\lambda^{-3}$ &  $0$ &  $0$ \\
 $\nu^{}_{01}$ &  $\lambda^{-3}$ &  $\lambda^{-2}$ & $0$ & 
 $\lambda^{-2}$ & $\lambda^{-2}$ &  $0$ & $0$ &  $0$ & 
 $\lambda^{-2}$ &  $0$ & $\lambda^{-2}$ \\
 $\nu^{}_{10}$ &  $\lambda^{-2}$ & $0$ &  $\lambda^{-2}$ &  $0$ &  
 $0$ &  $0$ & $\lambda^{-2}$ & $\lambda^{-2}$ & 
 $0$ & $\lambda^{-2}$ & $\lambda^{-3}$ \\
 $\nu^{}_{11}$ &  $3\lambda^{-4}$ & $0$ & $0$ & $\lambda^{-2}$ & 
 $2\lambda^{-2}$ & $3\lambda^{-2}$ & $2\lambda^{-2}$ & $\lambda^{-2}$ &  
 $0$ &  $0$ & $3\lambda^{-4}$ \\ \hline 
\end{tabular}
\end{table}

Next, let us observe that each pair correlation function can uniquely
be decomposed as $ \nu^{}_{ij} (z) = \nu^{+}_{ij} (z) + \nu^{-}_{ij}
(z)$, with
\[
   \nu^{\pm}_{ij} (z) \, := \, \myfrac{1}{2} \, \bigl(
   \nu^{}_{ij} (z) \pm \nu^{}_{ji} (-z)\bigr) ,
\]
into a symmetric and an anti-symmetric part. Since the right-hand
sides of the renormalisation equations are linear and preserve the
symmetry type, we may conclude as follows.

\begin{coro}
  Under the assumptions of Theorem~\emph{\ref{thm:reno-sols}}, all
  solutions of the renormalisation equations from
  Proposition~\emph{\ref{prop:reno-eqs}} automatically satisfy the
  symmetry requirement of Proposition~\emph{\ref{prop:reno-eqs}}. In
  particular, the only anti-symmetric solution is the trivial one.
  \qed
\end{coro}

Let us comment on these findings. While the renormalisation
equations\/ are a consequence of the structure of the hull $\YY$, it
is by no means obvious that their solutions are essentially
unique. This property was previously shown for the Fibonacci system in
\cite{BG15}, but need not hold in general. In fact, uniqueness fails
for the Thue--Morse system as soon as the support is enlarged, and
this is then related to the existence of a spectrum of mixed
type. To further analyse the (still more complex) situation in our
non-PV point set, we need to reformulate the above findings in
terms of measures and their Fourier transforms.

\section{Pair correlation measures and their Fourier 
transforms}\label{sec:reno-measure}
 
\subsection{From functions to measures}
Given the pair correlation functions $\nu^{}_{ij}$, which have the
locally finite point sets $S_{ij}$ as supports, we can turn them into
positive pure point measures by defining
\begin{equation}\label{def-ups}
     \vU_{ij} \, := \sum_{z \in S_{ij}} \nu^{}_{ij} (z)
     \, \delta^{}_{z} \ts .
\end{equation}
Together with Eq.~\eqref{eq:def-nuij}, for any $0 \leqslant i,j
\leqslant 1$, this implies the relation
\begin{equation}\label{eq:def-ups}
     \vU_{ij} \, = \,
    \frac{\widetilde{\delta^{}_{\!\!\vL^{(i)}}}\! \circledast 
    \delta^{}_{\!\!\vL^{(j)}}}{\dens (\vL)} \, = \,
     \frac{\delta^{}_{- \nts \vL^{(i)}}\! \circledast 
    \delta^{}_{\!\!\vL^{(j)}}}{\dens (\vL)}  \ts ,
\end{equation}
as can be verified by an explicit calculation that is analogous to
that for Eq.~\eqref{eq:Eberlein}. Here, as before, $\vL \in \YY$ is
arbitrary, but the result is independent of its choice due to
Fact~\ref{fact:support}.  Note that $\vU_{ij} \bigl( \{ x \} \bigr) =
\nu^{}_{ij} (x)$ in this notation, and that Eq.~\eqref{eq:def-ups}
gives a first hint on an underlying tensor product structure. For this
reason, as usual in multilinear algebra, we will sometimes view $\vU$
as a measure matrix, then written as $\bigl( \vU_{ij} \bigr)_{0
  \leqslant i,j \leqslant 1}$, or as a measure vector $( \vU_{00},
\vU_{01}, \vU_{10}, \vU_{11})^{t}$ with lexicographic index ordering,
whatever is better suited.

To expand on the last point, let us mention that the vector notation
will have some advantage in the formulation of the measure-valued
renormalisation relations to be derived shortly. On the other hand,
let us observe that the matrix version is certainly useful for lifting
Eq.~\eqref{eq:etanu} to the level of measures. If $(u^{}_{0},
u^{}_{1})$ are the (complex) weights of the measure $\omega$ from
Eq.~\eqref{eq:comb}, so $\omega = u^{}_{0} \, \delta^{}_{\!\!
  \vL_{\vphantom{a}}^{(0)}} + u^{}_{1} \, \delta^{}_{\!\!
  \vL_{\vphantom{a}}^{(1)}}$, the corresponding autocorrelation
measure $\gamma^{}_{u}$ can be written as
\begin{equation}\label{eq:new-gamma}
    \gamma^{}_{u} (\cE) \, = \, \dens (\vL) \!\! \sum_{i,j \in \{ 0,1\} }
    \!\! \overline{u^{}_{i}} \; \vU^{}_{ij} (\cE) \, u^{}_{j}
\end{equation}
where $\cE\subset\RR$ is any bounded Borel\footnote{Here, we use the
  general Riesz--Markov representation theorem that allows us to
  identify our measures in the sense of linear functionals on
  $C_{\mathsf{c}} (\RR)$ with regular Borel measures on $\RR$
  in the sense of Radon measures.} set. Note that
$\omega$ is defined via a fixed $\vL\in\YY$, but that $\gamma^{}_{u}$
does not depend on it by Proposition~\ref{prop:hull-2}.

Each $\vU_{ij}$ is a positive pure point measure that is unbounded. It
is translation bounded as a consequence of Eq.~\eqref{eq:def-ups},
because the Eberlein convolution of translation bounded measures is
translation bounded again. Moreover, each $\vU_{ii}$ is positive
definite, while the $\vU_{ij}$ for $i\ne j$ can be written as the
(complex) linear combination of four positive definite measures by an
application of the complex polarisation identity; compare
\cite[Lemma~1]{BG15}. Consequently, each $\vU_{ij}$ is Fourier
transformable as a measure \cite{Arga}, and we get the following
result from the Bochner--Schwartz theorem; see \cite{RS,BF,TAO} for
background.

\begin{fact}\label{fact:FT}
  Each pair correlation measure\/ $\vU_{ij}$ from
  Eq.~\eqref{eq:def-ups} is an unbounded, but translation bounded pure
  point measure with support\/ $S_{ij}$. Moreover, it is Fourier
  transformable as a measure, and the transform\/ $\widehat{\vU_{ij}}$
  is a translation bounded and positive definite measure, which is
  also positive if\/ $i=j$.  \qed
\end{fact}

At this point, we can Fourier transform Eq.~\eqref{eq:new-gamma} to
obtain
\begin{equation}\label{eq:u-FT}
    \widehat{\gamma^{}_{u}} (\cE) \, = \, \dens (\vL) \!\!
    \sum_{i,j \in \{ 0,1 \}} \!\! \overline{u^{}_{i}} \;
    \widehat{\vU}^{}_{ij} (\cE) \, u^{}_{j}
\end{equation}
as a general formula of the diffraction measure of the weighted Dirac
comb $\omega$ from Eq.~\eqref{eq:comb} in terms of the Fourier
transforms of the pair correlation measures, evaluated at a bounded
Borel set $\cE\subset \RR$. This explains why further relations among
the correlation measures will help to determine the spectral
properties of the diffraction measure.

\subsection{Renormalisation relations for measures}
Now, we have to rewrite the renormalisation equations for the pair
correlation functions from Proposition~\ref{prop:reno-eqs} in terms of
the correlation measures $\vU_{ij}$. To do so, we employ the
convolution with suitable \emph{finite} measures. This is well-defined
by standard results; see \cite[Prop.~1.13]{BF}. To proceed, recall
that we are working with prototiles (intervals) of natural lengths,
namely $\lambda$ and $1$ for tiles of type $0$ and $1$, respectively.
Define the set-valued location or \emph{displacement matrix}
$T = \bigl(T_{ij}\bigr)_{0\leqslant i,j \leqslant 1}$ by
\[
    T_{ij} \, := \, \{ \text{all relative positions of  tiles
     of type $i$ in the supertile of type $j$} \} \ts ,
\]
where the relative positions are again defined via the left endpoints
of the tiles (intervals). Clearly, $\card (T) := \bigl( \card (T_{ij})
\bigr)_{0\leqslant i,j \leqslant 1} = M$, and we have
\[
    T \, = \, \begin{pmatrix}
     \{ 0 \} & \{ 0 \} \\ 
     \{\lambda, \lambda + 1, \lambda + 2 \} &
     \varnothing \end{pmatrix} .
\]
We now turn this into a matrix of finite measures (or Dirac combs) by
setting
\[
     \delta^{}_{\ts T} \; := \;
     \begin{pmatrix} \delta^{}_{0} & \delta^{}_{0} \\
     \delta^{}_{\lambda} \nts + \delta^{}_{\lambda + 1}
     \nts + \delta^{}_{\lambda +2} & 0 \end{pmatrix}.
\]
Next, define the \emph{scaling function} (or dilation) $f$ on $\RR$ by
$f(z)=\lambda z$. Then, given $\mu$, the measure $f \nts. \ts \mu$ is
defined by $\bigl(f \nts . \ts \mu\bigr)(g) := \mu(g \circ f )$ for
$g\in C_{\mathsf{c}} (\RR)$. In particular, this means
$f \nts . \ts\ts \delta_{z} = \delta_{\nts f(z)}$. Using
$\delta_{x} * \delta_{y} = \delta_{x+y}$ and the usual rules for a
change of variable in weighted sums over Dirac measures, compare
\cite{BG15}, one obtains the following result from a straight-forward
calculation.

\begin{lemma}\label{lem:upsrec}
  With\/ $\vU_{ij}$ defined as in Eq.~\eqref{eq:def-ups}, the
  renormalisation equations from
  Proposition~\emph{\ref{prop:reno-eqs}} are equivalent to the measure
  identity
\[
       \vU \, = \, \myfrac{1}{\lambda} \, \Bigl(\ts
       \widetilde{\delta^{}_{T}} \stackrel{*}{\otimes} 
       \delta^{}_{T}  \Bigr) * (f. \ts\ts \vU) \ts ,
\]   
where\/ $\vU$ is considered as a vector of pure point measures, 
$*$ denotes convolution of
measures and\/ $\stackrel{*}{\otimes} $ the Kronecker convolution
product.  \qed
\end{lemma}   

More generally, if $\cM^{\infty} (\RR)$ denotes the space of
translation bounded measures on $\RR$, one can define a linear mapping
from $\cM^{\infty} (\RR)^{4}$ into itself by
$Y \mapsto \Bigl(\ts \widetilde{\delta^{}_{T}} \stackrel{*}{\otimes}
\delta^{}_{T} \Bigr) * (f\nts . Y)$,
which is continuous in the vague topology. Our measure vector $\vU\!$,
as defined with the pair correlation functions $\nu^{}_{ij}$ from
Proposition~\ref{prop:reno-eqs} and Theorem~\ref{thm:reno-sols}, is
then an eigenvector of this linear map, with eigenvalue $\lambda$.  It
might be an interesting question to analyse linear maps of this kind
more systematically. Here, with $\mu^{}_{\mathrm{L}}$ denoting
Lebesgue measure on $\RR$, we only mention that
 \begin{equation}\label{eq:sap}
    Y_{ij} \, := \, \nu^{}_{i} \, \nu^{}_{j}\, \mu^{}_{\mathrm{L}}
 \end{equation}
 with the frequencies $\nu^{}_{i}$ from Eq.~\eqref{eq:lettfreq} is
 another eigenvector. This can easily be checked from
 $\delta^{}_{x} * \mu^{}_{\mathrm{L}} = \mu^{}_{\mathrm{L}}$ for any
 $x\in\RR$ together with the observation that
 $v^{}_{\mathrm{PF}} \otimes v^{}_{\mathrm{PF}}$ is an eigenvector of
 $M\otimes M$, where $\otimes$ denotes the Kronecker product (see
 below for more).  In fact, up to an overall constant, the $Y_{ij}$
 from Eq.~\eqref{eq:sap} form the only solution of the measure-valued
 renormalisation relation where each $Y_{ij}$ is a multiple of
 Lebesgue measure. We will return to this point later.

\subsection{Renormalisation after Fourier transform}
Now, we can turn the new relation from Lemma~\ref{lem:upsrec} into one
on the Fourier side.  For consistency with previous work \cite{BG15},
we define
\[
       B \, = \, \widehat{ \widetilde{\delta^{}_{T}}}
       \, = \, \widehat{\delta^{}_{-T}} \, = \, 
       \overline{ \widehat{\delta^{}_{T}}} \ts ,
\]
where $B$ is a matrix of complex analytic functions because all
elements of $\delta^{}_{T}$ are finite measures with compact support;
$B$ is called the \emph{Fourier matrix} of the inflation rule.  In
line with common practice in this context, we use $k$ as the variable
on the Fourier side, which is real-valued for our further analysis.
Elementwise, we thus have
\begin{equation}\label{eq:B-def}
    B_{ij} (k) \, =  
   \sum_{t\in T_{ij}} \ee^{2\pi\ii  k t} ,
\end{equation}
where $B(0) = M$. Now, with 
\begin{equation}\label{eq:A-def}
       A(k) \, := \, B(k) \otimes \overline{B(k)},
\end{equation}
where the Kronecker product stands for the standard matrix
representation of the tensor product with lexicographic ordering of
the components, we can apply the convolution theorem to the identity
in Lemma~\ref{lem:upsrec}. Recalling further that
$\widehat{f \nts . \ts \mu} = \frac{1}{\lambda} \, ( f^{-1} \nts . \ts
\widehat{\mu} \ts )$ for our dilation $f$ and any transformable 
measure $\mu$, one obtains the following result.

\begin{prop}\label{prop:ups-FT}
  Under Fourier transform, the identity of Lemma~\emph{\ref{lem:upsrec}}
  turns into the relation
\[
       \widehat{\vU} \, = \, \myfrac{1}{\lambda^{2}} \,
       A(.) \cdot \bigl(f^{-1}\nts . \ts \widehat{\vU} \ts\ts \bigr),
\]    
   where\/ $A (.)$ is the matrix function of Eq.~\eqref{eq:A-def},
   $f$ is the dilation defined by\/ $f(x) = \lambda \ts x$, and\/
   $\widehat{\vU}$ is the vector of Fourier transforms of the pair
   correlation measures.  \qed
\end{prop}

In analogy to above, one may view $\widehat{\vU}$ as an eigenvector of
a linear mapping, this time with eigenvalue $\lambda^{2}$.  Spelled
out elementwise, with the double indices from the Kronecker product,
the identity in Proposition~\ref{prop:ups-FT} reads
\[
      \widehat{\vU}_{ij} \, = \sum_{m,n}  A_{ij,mn} (.) \,
      \bigl(f^{-1}\nts . \ts \widehat{\vU}_{mn}\bigr) \ts ,
\]
which also explains how the elements of the matrix $A$ appear as
densities for the entries of the measure vector $\bigl(f^{-1}. \ts
\widehat{\vU} \ts\ts \bigr)$.  Here, $A_{ij,mn} (k) = B_{im} (k) \,
\overline{B_{jn} (k)}$ from the Kronecker product structure. This
means that we can alternatively write the identity for $\widehat{\vU}$
in matrix form as
\begin{equation}\label{eq:ups-FT-matrix}
    \widehat{\vU} \, = \, \myfrac{1}{\lambda^{2}} \,
    B (.)  \bigl( f^{-1}\nts . \ts \widehat{\vU} \ts \ts 
    \bigr) B^{\dag} (.) \ts ,
\end{equation}
where $B^{\dag}$ denotes the Hermitian adjoint of $B$ and
$\widehat{\vU}$ is considered as a $2 \!\times\! 2$-matrix. Both
versions will be handy later on.

From this short derivation, is should be clear that the structure of
the matrix functions $B(k)$ and $A(k)$ are important. We thus turn to
a more detailed analysis of them, before we return to the
renormalisation identities and their consequences for our spectral
problem.

\section{Inflation displacement algebra and 
Kronecker product extension}\label{sec:algebra}

\subsection{The inflation displacement algebra}
Let us define the total set $S_{T} := \bigcup_{i,j} T_{ij}$ of all
relative locations of prototiles in level one supertiles, which means
\[
    S_{T} \, = \, \{ 0, \lambda, \lambda + 1, \lambda + 2 \} \ts .
\] 
Then, we can decompose the Fourier matrices $B(k)$ from
Eq.~\eqref{eq:B-def} as
\begin{equation}\label{eq:def-digit}
   B(k) \, =  \sum_{x\in S_{T}} \ee^{2\pi\ii k x} D_{x}
\end{equation}
with integer $0\ts\ts$-$1$-matrices $D_{x}$ that satisfy $\sum_{x\in
  S_{T}} D_{x} = M$. Explicitly, we have
\begin{equation}\label{eq:dig-mat}
   D_{0} \, = \, \begin{pmatrix} 1 & 1 \\ 0 & 0 \end{pmatrix}
   \quad \text{and} \quad D_{\lambda} \, = \, D_{\lambda+1} \, = \,
   D_{\lambda+2} \, = \, \begin{pmatrix} 0 & 0 \\ 1 & 0 \end{pmatrix},
\end{equation}
where the matrix elements are given by
\[
   D_{x, ij} \, = \, \begin{cases}
   1 , & \text{if the supertile of type $j$ contains a tile of
    type $i$ at $x$} , \\
   0 , & \text{otherwise}\ts .  \end{cases}
\]
These geometric incidence matrices are a generalisation of what is
known as \emph{digit matrices} in constant length substitutions
\cite{Vince,Nat1}, wherefore we adopt this terminology here as well;
compare also \cite[Sec.~8.1]{Q}. Recently, also the term
\emph{instruction matrices} has been used \cite{Bartlett}.

Next, consider the $\CC$-algebra $\cB$ that is generated by the
one-parameter matrix family $\{ B(k) : k\in \RR \}$. We call this
finite-dimensional algebra, which is automatically closed, the
\emph{inflation displacement algebra} (IDA) of $\varrho$.  Consider
also the $\CC$-algebra $\cB_{\nts D}$ that is generated by the digit
matrices $\{ D_{x} : x\in S_{T}\}$, hence by $D_{0}$ and
$D_{\lambda}$ in our case. Recall that a family of
$2\!\times\!  2$-matrices is called \emph{irreducible} (over the field
$\CC$) if the only subspaces of $\CC^{2}$ that are invariant under all
elements of the matrix family are the trivial subspaces, $\{ 0 \}$ and
$\CC^{2}$.

\begin{lemma}\label{lem:ida}
  The IDA of\/ $\varrho$ satisfies\/
  $\cB = \cB_{\nts D} = \mathrm{Mat} (2,\CC)$, which is also the IDA
  of\/ $\varrho^{n}$ for any\/ $n\in\NN$.  In particular,
  $\cB_{\nts D}$ is irreducible. For any\/ $\varepsilon > 0$, the
  complex algebra generated by the matrix family\/
  $\{ B(k) : 0 \leqslant k < \varepsilon \}$ is again\/ $\cB$. In fact,
  any finite set of matrices\/ $\{ B(k) : k\in J \}$ generates\/
  $\cB$, if\/ $J$ contains at least two elements at which the
  trigonometric polynomial\/
  $p(k) = \ee^{2 \pi \ii k \lambda} + \ee^{2 \pi \ii k (\lambda + 1)}
  + \ee^{2 \pi \ii k (\lambda + 2)} = \ee^{2 \pi \ii k (\lambda + 1)}
  \bigl( 1 + 2 \ts \cos (2 \pi k) \bigr)$ differs.
\end{lemma}

\begin{proof}
  From Eq.~\eqref{eq:dig-mat}, we already know that the digit matrices
  for $\varrho$ are $D_{0} = E_{00} + E_{01}$ and
  $D_{\lambda} = E_{10}$ with $E_{ij}$ denoting the standard
  elementary matrices, indexed with $i,j \in \{ 0,1 \}$ to match the
  labelling of the tile types.  Observing that
  $D_{0}\ts D_{\lambda} = E_{00}$ and
  $D_{\lambda} D_{0} = E_{10} + E_{11}$, we get $E_{01}$ and $E_{11}$
  via differences. We thus have all four elementary matrices $E_{ij}$
  within $\cB_{\nts D}$. This implies
  $\cB_{\nts D} = \mathrm{Mat}(2,\CC)$, which is irreducible.

  Next, one checks that $D_{0}$ and $D_{\lambda}$ are always among the
  digit matrices for $\varrho^{n}$ with $n\in\NN$, hence we always get
  the full matrix algebra, $ \mathrm{Mat}(2,\CC)$.

  The last claim, and then also the previous one, follows from the
  observation that Eq.~\eqref{eq:def-digit}, in view of the second
  identity in Eq.~\eqref{eq:dig-mat}, can be rewritten as
  $B (k) = D_{0} + p(k) \ts D_{\lambda}$. Thus, knowing $B(k)$ for
  $k_1$ and $k_2$ with $p(k_1) \ne p(k_2)$, we have a system of two
  equations that can be solved for $D_{0}$ and $D_{\lambda}$.  This
  also implies $\cB_{\nts D} \subseteq \cB$, while
  $\cB \subseteq \cB_{\nts D}$ is clear from Eq.~\eqref{eq:def-digit},
  so $\cB = \cB_{\nts D}$ as stated.
\end{proof}

For later reference, we note that the matrix $B(k)$ can now be written
as
\begin{equation}\label{eq:B-explicit}
   B(k) \, = \, \begin{pmatrix} 1 & 1 \\ p(k) & 0 \end{pmatrix},
\end{equation}
where $p$ is the trigonometric polynomial from
Lemma~\ref{lem:ida}. Let us state an elementary observation that will
later give us access to a compactness argument; see
\cite[Ex.~8.1]{TAO} and references given there for background
on quasiperiodic functions.

\begin{fact}\label{fact:quasi}
  The polynomial\/ $p$ from Lemma~\mbox{$\ts \ref{lem:ida}$} is
  quasiperiodic, with fundamental frequencies\/ $1$ and\/
  $\lambda$. As such, it can be represented as
\[
       p (k) \, = \, \tilde{p} (x,y)\big|_{x=\lambda k, \, y=k} 
\]   
with\/ $\tilde{p} (x,y) := \ee^{2 \pi \ii (x+y)} \ts 
\bigl(1 + 2 \cos (2 \pi y) \bigr)$,
which is\/ $1$-periodic in both arguments. In this 
representation, one has
\[
    p (\lambda k) \, = \, \tilde{p} \bigl( (x,y) 
    M \bigr)\big|_{x=\lambda k, \, y=k} \ts = \,
    \tilde{p} (x+3 y, x)\big|_{x=\lambda k, \, y=k}
\]      
   with the substitution matrix\/ $M$ from Eq.~\eqref{eq:submat}.
   Likewise, $B (k)$ defines a quasiperiodic matrix function, with\/
   $B (k) = \tilde{B} (x,y)\big|_{x=\lambda k, \, y=k}$ and\/ $\tilde{B}
   (x,y) = \left( \begin{smallmatrix} 1 & 1 \\ \tilde{p} (x,y) & 0
   \end{smallmatrix}\right)$.
\end{fact}

\begin{proof}
  The first claim is clear, while the second follows from the fact
  that $(\lambda , 1)$ is the left PF eigenvector of $M$; this
  relation motivated the particular representation we chose.

The consequence for $B$ is now obvious.
\end{proof}

\subsection{Kronecker products}
Let us consider the matrix function
$A (k) = B(k) \otimes \overline{B(k)}$ introduced above in
Eq.~\eqref{eq:A-def}, with $B(k)$ as in Eq.~\eqref{eq:B-explicit}.
This somewhat unusual product with complex
conjugation reflects, on the Fourier side, the symmetry relation
$\nu^{}_{ji} (z) = \nu^{}_{ij} (-z)$ from
Proposition~\ref{prop:reno-eqs}.  So, consider the matrix family
$\{ {A} (k) : k\in\RR\}$ or any (possibly finite, but sufficiently
large) subset of it. Even though the IDA $\cB$ is irreducible, this
does \emph{not} transfer to the ${A}$-matrices, as is well-known from
representation theory.  Indeed, let $V=\CC^{2}$ and consider
$W:= V\otimes^{}_{\CC} V$, the (complex) tensor product. $W$ is a
vector space over $\CC$ of dimension $4$, but also one over $\RR$,
then of dimension $8$.

Now, view $W$ as an $\RR$-vector space and consider the involution
$C\!  : \, W \longrightarrow W$ defined by
\[
   x \otimes y \; \longmapsto \; C (x\otimes y)
   := \overline{y\otimes x} = \bar{y} \otimes \bar{x} 
\]
together with its unique extension to an $\RR$-linear mapping on $W$.
Note that there is no $\CC$-linear extension, because $C\bigl( a (x\otimes
y)\bigr) = \bar{a}\, C(x\otimes y)$ for $a\in\CC$.  From the
definition, one finds
\[
\begin{split}
  {A} (k)\, C(x\otimes y) \, & = \, 
   \bigl(B(k) \otimes \,\overline{\! B(k)\! }\, \bigr)
   ( \bar{y} \otimes \bar{x} ) \, = \, 
    \bigl(B(k)\ts \bar{y}\bigr) \otimes 
    \bigl(\, \overline{\! B(k)\ts  x \nts} \, \bigr) \\
   & = \, C \left(   \bigl(B(k) \otimes \, 
       \overline{\! B(k)\! }\, \bigr)
       (x \otimes y ) \right) \, = \, 
    C \bigl( {A} (k) \, ( x \otimes y ) \bigr),
\end{split}
\]
so $C$ commutes with the linear map defined by ${A} (k)$, for
any $k\in\RR$.  The eigenvalues of $C$ are $\pm 1$, and our
vector space splits as $W \! = W_{\! +} \oplus W_{\! -}$ into real
vector spaces that are the eigenspaces of $C$, so $W^{}_{\!
  \pm} = \{ x \in W : C(x) = \pm\ts x \}$.  Their dimensions
are
\[
   \dim^{}_{\RR} (W_{\! +}) \, = \,  \dim^{}_{\RR} (W_{\! -}) \, = \, 4
\]
since $W_{\! -} = \ii \ts W_{\! +}$ with $W_{\! +} \cap W_{\! -} = \{
0 \}$.  The corresponding splitting of an arbitrary $w\in W$ is unique
and given by $w = \frac{1}{2} \bigl( w + C(w)\bigr) + \frac{1}{2}
\bigl( w - C(w)\bigr)$ as usual.

It is now clear that $W_{\! +}$ and $W_{\! -}$ are invariant (real)
subspaces of our matrix family, and of the $\RR$-algebra ${\cA}$
generated by it. Fortunately, due to the symmetry relation for the
correlation coefficients, the irreducibility of $\cA$ on the subspace
$W_{\! +}$ is what matters later on.

Let us pause to observe a connection with the digit matrices.
We have
\[
   {A} (k) \, = \, B(k) \otimes \overline{B (k)} \; = \!
   \sum_{x,y \in S_{T}}\! \ee^{2 \pi \ii k (x-y)} \, D_{x} \otimes D_{y}
   \; = \sum_{z\in S_{T} - S_{T}}  \ee^{2 \pi \ii k z}\, F_{z}
\]
where 
\[
   F_{z} \; = \sum_{\substack{x,y \in S_{T} \\ x-y = z}} 
     D_{x} \otimes D_{y} \ts .
\]
In particular, one has
\[
   F^{}_{0} \, = \, \begin{pmatrix} 1 & 1 & 1 & 1 \\
   0 & 0 & 0 & 0 \\ 0 & 0 & 0 & 0 \\ 3 & 0 & 0 & 0 \end{pmatrix} 
\]
with spectrum $\{\lambda,\lambda',0,0\}$, while all $F_{z}$ with $0\ne
z\in S_{T}- S_{T}$ are nilpotent, with $F_{z}^{2}=0$.  Now, one easily
checks that $\bigl(C (a\ts F_{z})\ts C\bigr) (u\otimes v) = \bar{a}\,
F_{-z} (u\otimes v)$, which implies $\bigl[C, {A}(k)\bigr]=0$ for all
$k\in\RR$, in line with our previous derivation.

\begin{lemma}\label{lem:irred}
  The\/ $\RR$-algebra\/ ${\cA}$ satisfies\/ $\dim^{}_{\RR}
  ({\cA}) = 16$ and acts irreducibly on each of the invariant
  four-dimensional subspaces\/ $W_{\! +}$ and\/ $W_{\! -}$ 
  introduced above.
\end{lemma}

\begin{proof}
  We will show that ${\cA} \simeq \mathrm{Mat} (4,\RR)$.  If one
  considers $\mathrm{Mat}(4,\CC)$ as an $\RR$-algebra of dimension
  $32$, the subalgebra ${\cA}$ consists of all elements that are fixed
  under the $\RR$-linear mapping defined by
  $M_{1}\otimes M_{2}\mapsto \overline{M_{2}}\otimes\overline{M_{1}}$.
  Consequently, with $E_{ij,k\ell} := E_{ik}\otimes E_{j\ell}$, a
  basis of the $\RR$-algebra ${\cA}$ as a $16$-dimensional real vector
  space can be constructed from the spanning set
\[
   \bigl\{\tfrac{1}{2}(E_{ij,k\ell}+E_{ji,\ell k}): i,j,k,\ell
    \in \{ 0,1 \} \bigr\} \cup 
   \bigl\{\tfrac{\ii}{2}(E_{ij,k\ell}-E_{ji,\ell k}): 
    i,j,k,\ell \in \{ 0,1 \} \bigr\},
\]
where the first resp.\ second subset contains $10$ resp.\ $6$ linearly
independent elements.

Next, consider the unitary matrix
\begin{equation}\label{eq:umat}
    U \, = \, \frac{1}{\sqrt{2}}
    \begin{pmatrix}1\!-\!\ii & 0 & 0 & 0 \\ 0 & 1 & -\ii & 0 \\
     0 & -\ii & 1 & 0 \\ 0 & 0 & 0 & 1\!-\!\ii\end{pmatrix},
\end{equation}
which acts on $\mathrm{Mat}(4,\CC)$ via conjugation $(.)\mapsto
U(.)\ts\ts U^{-1}$. When applied to the basis matrices of ${\cA}$, one
obtains $16$ real matrices which, by suitable integer linear
combinations, yield all $16$ elementary matrices
$E_{ij,k\ell}$. Therefore, within $\mathrm{Mat}(4,\CC)$, ${\cA}$ is
conjugate to $\mathrm{Mat}(4,\RR)$. This implies that the action of
${\cA}$ on the invariant subspaces $W_{\! +}$ and $W_{\! -}$ is
irreducible.
\end{proof}

\subsection{Consequences}
Observe next that $[U,{A}(0)]=0$ and that $U(W^{}_{\!\mp})= \frac{1
  {\scriptscriptstyle\pm} \ii}{\sqrt{2}}\, \RR^{4}$. Moreover,
\[
    {A}^{}_{U}(k) \, := \, 
    U {A}(k)\, U^{-1} \, = \, \begin{pmatrix}
    1 & 1 & 1 & 1\\
    c(k)+s(k) & s(k) & c(k) & 0\\
    c(k)-s(k) & c(k) & -s(k) & 0\\
    c(k)^2+s(k)^2 & 0 & 0 & 0 \end{pmatrix}
\]
with
\[
\begin{split}
   c(k) & \,=\, \cos(2\pi \lambda k) + 
     \cos(2\pi (\lambda \! + \! 1) k) + \cos(2\pi
   (\lambda \! + \! 2) k)\ts ,\\
  s(k)&\, =\, \sin(2\pi \lambda k) \ts +\ts 
     \sin(2\pi (\lambda \! + \! 1)k) \ts +\ts 
  \sin(2\pi (\lambda \! + \! 2) k)\ts ,
\end{split}
\]
so that $c(k)^2+s(k)^2= \bigl| p(k) \bigr|^{2} = \bigl(1+2\cos(2\pi
k)\bigr)^2$ and ${A}^{}_{U}(0)={A}(0)=M\otimes M$ with the
substitution matrix $M$ from Eq.~\eqref{eq:submat}. Since
${A}^{}_{U}(0)^{2}$ is a strictly positive integer matrix, and since
${A}^{}_{U} (\frac{k}{\lambda})\ts {A}^{}_{U}(k)$ is continuous in
$k$, the latter product is strictly positive as well for sufficiently
small $k$. As one can easily calculate numerically, the smallest $k>0$
for which one element of ${A}^{}_{U}(\frac{k}{\lambda}) \ts
{A}^{}_{U}(k)$ vanishes (and strict positivity is thus violated) is $k
\approx 0.03832 \ts (1)$, so the product is certainly a strictly
positive matrix for all $k\in [0,\varepsilon]$ with $\varepsilon =
0.03$ say.  Moreover, ${A}^{}_{U}(k)$ cannot have zero eigenvalues for
any $k$ with $\lvert k\rvert<\frac{1}{3}$, as follows from a simple
determinant argument (see the proof of Lemma~\ref{lem:eps-suffice}
below for details).

\begin{prop}\label{prop:simple-explode}
  Let\/ $k\in [0,\varepsilon]$ with this choice of\/ $\varepsilon$,
  and consider the iteration
\[
   w^{}_{n} \, := \,
   \Bigl({A}^{}_{U}\bigl(\tfrac{k}{\lambda^{2n-1}}\bigr)  \ts
   {A}^{}_{U}\bigl(\tfrac{k}{\lambda^{2n-2}}\bigr) \Bigr) \cdot \ldots 
   \cdot \Bigl( {A}^{}_{U}\bigl(\tfrac{k}{\lambda}\bigr) \ts
   {A}^{}_{U}\bigl(k\bigr) \Bigr)\ts w^{}_{0}
\]
for\/ $n\geqslant 1$ and any non-negative starting vector $w^{}_{0}
\ne 0$. Then, the vector $w^{}_{n}$ will be strictly positive for
all\/ $n\in\NN$ and, as\/ $n\to\infty$, it will diverge with asymptotic
growth\/ $c \ts \lambda^{4n} \ts w^{}_{\mathrm{PF}}$. Here, $c$ is a
constant that depends on\/ $w^{}_{0}$ and\/ $k$, while\/
$w^{}_{\mathrm{PF}} = v^{}_{\mathrm{PF}} \otimes v^{}_{\mathrm{PF}}$
is the statistically normalised PF eigenvector of\/ $M\otimes M$
derived from Eq.~\eqref{eq:lettfreq}.
\end{prop}

\begin{proof}
  For $k=0$, this is standard Perron--Frobenius theory for the
  non-negative and primitive matrix ${A} (0) = M \otimes M$, which has
  PF eigenvalue $\lambda^{2}$ with eigenvector
  $w^{}_{\mathrm{PF}}$. In particular, $ A (0)^2$ is
  strictly positive.  Since $w^{}_{0}$ must have a component in the
  direction of $w^{}_{\mathrm{PF}}$, this component leads to the
  asymptotic behaviour claimed.

  For $0 < k \leqslant \varepsilon$, one iterates with strictly
  positive matrices due to the choice of $\varepsilon$, so
  $w^{}_{0} \ne 0$ means that $w^{}_{1}$ is already strictly
  positive. Clearly,
  ${A}^{}_{U} (\frac{k}{\lambda^{2n-1}}) \ts {A}^{}_{U}
  (\frac{k}{\lambda^{2n-2}})$
  is analytic in $k$ and converges to ${A} (0)^{2}$ as $n\to \infty$.
  No cancellation can occur in the iteration and, asymptotically, each
  iteration step multiplies the component in the direction of
  $w^{}_{\mathrm{PF}}$ by $\lambda^{4}$. Due to the existence of a
  spectral gap for ${A} (0)^{2}$, this component dominates, and the
  claim follows.
\end{proof}

Let us formulate a simple consequence that will be useful in our
later analysis.

\begin{coro}\label{coro:simple-explode}
  Under the assumptions and in the setting of
  Proposition~\emph{\ref{prop:simple-explode}}, consider the iteration
\[
    w^{\ts\prime}_{m} \, = \, \myfrac{1}{\lambda} \ts
    {A}^{}_{U} \bigl( \tfrac{k^{}_{0}}{\lambda^{m}} \bigr)
    \, w^{\ts\prime}_{m-1}
\]
for\/ $m \in \NN$, with\/ $k^{}_{0} \in [0,\varepsilon]$ and\/
$w^{\ts\prime}_{0} = w^{}_{0}$.  Then, $w^{\ts\prime}_{m} \sim c\ts
\lambda^{m} \, w^{}_{\mathrm{PF}}$ as\/ $m \to \infty$, where the
constant\/ $c$ depends on\/ $w^{}_{0}$ and\/ $k^{}_{0}$ as before. In
other words, when viewing\/ $w^{\ts \prime}_{m}$ as a function\/
$w^{\ts \prime}$ evaluated at\/ $k = k^{}_{0}/\lambda^{m}$, one has
\[
      w^{\ts \prime} (k) \, \sim \, \myfrac{c}{k} \,
      w^{}_{\mathrm{PF}}
      \quad \text{as $k \to 0$}
\]
along the sequence\/ $\bigl( k^{}_{0}/ \lambda^{m} 
\bigr)_{m\in \NN_{0}}$.  \qed
\end{coro}

We now have the necessary tools to return to the analysis of the pair
correlation measures and their Fourier transforms.

\section{Renormalisation analysis of pair correlation 
measures}\label{sec:explode}

\subsection{Lebesgue decomposition and consequences}
The measure vector $\widehat{\vU}$ satisfies the renormalisation
equation from Proposition~\ref{prop:ups-FT}. Each measure
$\widehat{\vU}_{ij}$ has a unique Lebesgue decomposition
\[
   \widehat{\vU}_{ij} \, = \, 
   \bigl(\widehat{\vU}_{ij}\bigr)_{\mathsf{pp}} +
   \bigl(\widehat{\vU}_{ij}\bigr)_{\mathsf{cont}} \, = \,
   \bigl(\widehat{\vU}_{ij}\bigr)_{\mathsf{pp}} +
   \bigl(\widehat{\vU}_{ij}\bigr)_{\mathsf{sc}} +
   \bigl(\widehat{\vU}_{ij}\bigr)_{\mathsf{ac}} 
\]
into a pure point and a continuous part, where the latter can further
be decomposed into a singular continuous and an absolutely continuous
part relative to Lebesgue measure.  By standard arguments, this can be
extended to the vector measure $\widehat{\vU}$ in such a way that the
supporting sets of the three parts coincide for all
$\widehat{\vU}_{ij}$.

\begin{lemma}\label{lem:spectral-parts}
  Each spectral component of\/ $\widehat{\vU}$ satisfies the
  renormalisation relation of Proposition~\emph{\ref{prop:ups-FT}}
  separately, which means that
\[
    \bigl(\widehat{\vU}\ts\ts \bigr)_{\mathsf{t}} \; = \; 
    \myfrac{1}{\lambda^2}\, {A} (.) \cdot  
    \bigl(f^{-1}.\,(\widehat{\vU}\ts )^{}_{\mathsf{t}}\bigr)
\]
holds for each spectral type\/ $\mathsf{t} \in \{ \mathsf{pp},
\mathsf{sc}, \mathsf{ac} \}$.
\end{lemma}

\begin{proof}
  Observe that the dilated measure $f^{-1}\nts . \ts \mu$ has the same
  spectral type as $\mu$, and that the matrix function ${A} (k)$ is
  analytic, so cannot mix different spectral types, which are mutually
  orthogonal in the measure-theoretic sense; compare
  \cite[Prop.~8.4]{TAO}.  Consequently, for any $\mathsf{t} \in \{
  \mathsf{pp}, \mathsf{sc}, \mathsf{ac} \}$, one has
\[
   \Bigl( \myfrac{1}{\lambda^2}\, {A} (.) \cdot  
    \bigl(f^{-1}\nts .\ts \widehat{\vU}
    \ts\ts \bigr)\!\Bigr)_{\! \mathsf{t}}
   \, = \, \myfrac{1}{\lambda^2}\, {A} (.) \cdot  
    \left(f^{-1}\nts .\,(\widehat{\vU}\ts 
    )^{}_{\mathsf{t}}\right).
\]
The claim now follows from the linearity of the renormalisation
relation for $\widehat{\vU}$.
\end{proof}

Let us next observe that we have $\widehat{\vU}_{ij} =
\widehat{\vU_{ij}}$ by definition, and that $\bigl( \widehat{\vU}_{ij}
\bigr)$ is Hermitian as a matrix, because
\[
    \overline{\widehat{\vU_{ij}}} \, = \,  
    \widehat{\widetilde{\vU_{ij}}} \, = \,  
    \widehat{\vU_{ji}} \ts .
\]
If we combine this with Eq.~\eqref{eq:u-FT} for
$\widehat{\gamma^{}_{u}}$, which is a positive measure for any complex
weight vector $u$, we get the following property.

\begin{fact}\label{fact:Hermitian}
  For any bounded Borel set\/ $\cE\subset \RR$, the complex matrix\/ $
  \bigl( \widehat{\vU}_{ij} (\cE) \bigr)_{0 \leqslant i,j \leqslant
    1}$ is Hermitian and positive semi-definite. Since\/
  $\widehat{\vU}^{}_{\!00}$ and\/ $\widehat{\vU}^{}_{\!11}$ are
  positive measures, positive semi-definiteness is equivalent to the
  determinant condition\/ $ \det \bigl( \widehat{\vU}_{ij} ( \cE)
  \bigr)^{\vphantom{I}} \geqslant 0$.  \qed
\end{fact}

\subsection{Renormalisation for pure point part}
Let us further explore the meaning of the relation from
Lemma~\ref{lem:spectral-parts}.  Clearly, the pure point part of
$\widehat{\vU}$ is of the form
\[
  \bigl(\widehat{\vU} \ts\ts \bigr)_{\mathsf{pp}} \, = \,
   \sum_{k\in K} \cI (k) \ts \delta^{}_{k} \ts ,
\]
with $ \cI (k) = \widehat{\vU} \bigl( \{ k \} \bigr)$.  Here, $\bigl(
\cI_{ij} (k) \bigr)_{0 \leqslant i,j \leqslant 1}$ is Hermitian and
positive semi-definite as a consequence of Fact~\ref{fact:Hermitian},
and $K$ is at most a countable subset of $\RR$, where $\lambda K
\subset K$ may be assumed. Inserting   
$\bigl(\widehat{\vU} \ts\ts \bigr)_{\mathsf{pp}} $
into the relation from Lemma~\ref{lem:spectral-parts} leads, after 
some calculations, to the relation
\begin{equation}\label{eq:intens-reno}
    \cI (k) \, = \, \myfrac{1}{\lambda^{2}} \ts
    {A} (k) \, \cI (\lambda k) \ts ,
\end{equation}
which has to hold for all $k\in K$. In particular, for $k=0$, one has
\[
    {A}(0) \, \cI (0) \, = \,  \lambda^{2} \, \cI (0) \ts ,
\]
where $\lambda^{2}$ is the PF eigenvalue of ${A} (0) = M
\otimes M$. This, with Eq.~\eqref{eq:dens-rel}, implies
\begin{equation}\label{eq:zero-peak}
   \widehat{\vU}_{ij} \bigl( \{ 0 \} \bigr)
   \, = \, \cI_{ij} (0) \, = \, \nu^{}_{i} \, \nu^{}_{j} \, = \,
   \frac{\dens (\vL^{(i)}) \ts \dens (\vL^{(j)})}
       {\bigl(\dens (\vL)\bigr)^{2}} \ts ,
\end{equation}
where we employed the PF eigenvector of $A(0)$ together with the
normalisation condition $\sum_{i,j} \cI_{ij} (0) = 1$, the latter
being a consequence of our setting with relative frequencies;
compare~\cite[Cor.~9.1]{TAO} for the underlying calculation. Note that
Eq.~\eqref{eq:zero-peak} also means that the Hermitian
$2 \! \times \!  2$-matrix $\bigl( \cI_{ij} (0) \bigr)$ has rank $1$.
Moreover, Theorem~\ref{thm:pp} implies that $\cI (k) = 0$ for all
$k\ne 0$, because otherwise Eq.~\eqref{eq:u-FT} would give
$\widehat{\gamma^{}_{u}} \bigl( \{ k \} \bigr) = \dens (\vL^{w})
\sum_{i,j} \overline{u^{}_{i}} \; \cI^{}_{ij} (k) \, u^{}_{j} > 0$
for a suitable choice of the weights, in contradiction to
Theorem~\ref{thm:pp}.  Consequently, we have
\begin{equation}\label{eq:I-pointpart}
    \bigl(\widehat{\vU}\ts\bigr)_{\mathsf{pp}} 
    \, = \, \cI (0) \, \delta^{}_{0}
\end{equation}
and the pure point part of $\widehat{\vU}$ is completely determined
this way.

\begin{remark}\label{rem:Eberlein}
  Let us relate Eq.~\eqref{eq:I-pointpart} back to the measure vector
  $Y$ from Eq.~\eqref{eq:sap}. With hindsight, $Y$ is the strongly
  almost periodic part of the \emph{Eberlein decomposition}
  \cite{GLA,MS} of $\vU$, so that
  $\widehat{Y} = \bigl(\widehat{\vU}\ts\bigr)_{\mathsf{pp}} $. Note
  that any $Y_{ij}$ is an absolutely continuous measure with support
  $\RR$, and thus \emph{not} a pure point measure as in the case of
  the Fibonacci chain in \cite{BG15}. This is related to the set
  $\varDelta$ from Proposition~\ref{prop:hull} not being uniformly
  discrete, and is the deeper reason why our discrete equations in
  Proposition~\ref{prop:reno-eqs} cannot produce a spectral purity
  result.  \exend
\end{remark}

\subsection{Analysis of absolutely continuous part}
Next, let $h_{ij}$ be the Radon--Nikodym density for
$\bigl(\widehat{\vU}_{ij}\bigr)_{\mathsf{ac}}$, and $h$ the
corresponding vector of densities. Each component is a locally
integrable function on the real line.  Another application of
Lemma~\ref{lem:spectral-parts}, together with an elementary
transformation of variable calculation, which changes the scalar
prefactor, then reveals the identity
\begin{equation}\label{eq:h-iter}
   {h} \bigl(\tfrac{k}{\lambda} \bigr) \, = \,
   \myfrac{1}{\lambda} \ts {A} \bigl(\tfrac{k}{\lambda}
   \bigr) \, {h} (k) \ts ,
\end{equation}
which must hold for Lebesgue-almost every (a.e.) $k\in \RR$. An
iteration gives
\begin{equation}\label{eq:h-iter-mult}
   {h} \bigl(\tfrac{k}{\lambda^{n}}\bigr) \, = \,
   \myfrac{1}{\lambda^{n}} {A} 
   \bigl(\tfrac{k}{\lambda^{n}} \bigr)
   \cdot \ldots \cdot {A} 
   \bigl(\tfrac{k}{\lambda}\bigr) \, {h} (k)
\end{equation}
for any $n\in\NN$ and still almost every $k\in\RR$. We thus obtain a
matrix Riesz product type consistency equation for ${h}$, which can
also be viewed as a (complex) linear \emph{cocycle} for the mapping 
defined by $k \mapsto k/\lambda$ on $\RR_{+}$.

Let us pause to explain the idea behind our ensuing
analysis. Iterating a vector $h(k)$ inwards via
Eq.~\eqref{eq:h-iter-mult} means that, asymptotically, we multiply
with a matrix from the left that more and more looks like
$A(0) = M \otimes M$. Since the leading eigenvalue is then (almost)
$\lambda^{2}$, and we divide only by a single factor of $\lambda$, a
component in the direction of $w^{}_{\mathrm{PF}}$ (provided it
exists) gets approximately multiplied by $\lambda$ in each iteration
step.  Consequently, we get a `blow-up' of $h(k)$ along the (inward)
iteration sequence, which means an asymptotic growth as $k^{-1}$ for
$k \searrow 0$. Clearly, such a behaviour is impossible for a locally
integrable function; see also Appendix B for an illustration in a
one-dimensional analogue.

There are two mechanisms to avoid the blow-up: Either $h$ vanishes for
a.e.\ $k\in\RR$, or $h(k)$ is of a form that asymptotically avoids the
direction of $w^{}_{\mathrm{PF}}$. The first option is what we are
after, while the second is connected with the Lyapunov spectrum of our
iteration. As we shall see, the second possibility still occurs, and
must be ruled out by the analysis of the asymptotic behaviour in an
inverted (outward) iteration, so that we finally get $h=0$ in the
Lebesgue sense, which means
$\bigl(\widehat{\vU}\ts\ts \bigr)_{\mathsf{ac}} = 0$.  The actual line
of arguments to make this strategy work is a bit delicate and
technical, and will require a number of steps.\smallskip

The rough outline of our arguments is as follows. The goal is to show
that $h$ vanishes, and we only need to do so on a small interval
(Lemma~\ref{lem:eps-suffice}). With the extra hermiticity and rank
structure of the matrix $(h^{}_{ij})$ from Lemma~\ref{lem:rank-one},
we gain a dimensional reduction, via Fact~\ref{fact:rank-1}. This
allows us to work with the Fourier matrices rather than their
Kronecker products, where we profit from the algebraic structure
derived earlier. In Proposition~\ref{prop:det-cond}, we derive an
asymptotic result on the determinants for the iterated application of
our recursion that will later lead to an important relation between
the Lyapunov exponents of our iteration.

The next step consists in the analysis of the inward iteration, which
gives interesting insight (Proposition~\ref{prop:inward}), but does
not suffice to derive $h=0$. The main step then consists in deriving
the (pointwise) Lyapunov exponents for the outward iteration
(Proposition~\ref{prop:L-out}) and to show that they are both strictly
positive (Corollary~\ref{coro:positive}), which is incompatible with
the translation boundedness of the diffraction measure. This will then
lead to the desired conclusion that $h=0$ in the Lebesgue sense
in Theorem~\ref{thm:no-ac}, and to the determination of
 $\widehat{\vU}_{ij}$ in
Corollary~\ref{coro:singular}. En route, we need a number of
intermediate steps that revolve around the existence of various
limits, where we need methods from Diophantine approximation,
matrix cocycles, and the theory of almost periodic functions.

\begin{lemma}\label{lem:eps-suffice}
  Let\/ ${h}$ be the vector of Radon--Nikodym densities of\/
  $(\widehat{\vU}\ts )^{}_{\mathsf{ac}}$.  If there is an\/
  $\varepsilon > 0$ such that\/ ${h} (k) = 0$ for a.e.\ $k\in \bigl[
  \frac{\varepsilon}{\lambda}, \varepsilon \bigr]$, one has\/ ${h} (k)
  = 0$ for a.e.\ $k\in\RR$.
\end{lemma}

\begin{proof}
  A multiple application of Eq.~\eqref{eq:h-iter} implies that ${h}
  (k) = 0$ for a.e.\ $k\in \bigl[ \frac{\varepsilon}{\lambda^{m+1}},
  \frac{\varepsilon}{\lambda^{m}} \bigr]$, for any $m \in
  \NN^{}_{0}$. Consequently, ${h} (k) = 0$ for a.e.\ $k\in
  [0,\varepsilon]$.

  Since $\det\bigl({A}(k)\bigr)= \bigl|\det\bigl(B(k)\bigr) \bigr|^{4}
  = \bigl| p(k) \bigr|^{4}
  = \bigl( 1 + 2 \cos(2 \pi k) \bigr)^{\! 4} \geqslant 0 $, the matrix
  ${A}(k)$ is invertible, unless $k \in Z:= \ZZ + \bigl\{ \frac{1}{3} ,
  \frac{2}{3} \bigr\} = \frac{1}{3} \ZZ \setminus \ZZ$. The latter is
  a countable set of isolated points, and thus of measure $0$. This
  implies that Eq.~\eqref{eq:h-iter} has a counterpart of the form
\begin{equation}\label{eq:h-iter-two}
    {h} ( \lambda k) \, = \, \lambda\ts
    {A}^{-1} (k) \, {h} (k) \ts ,
\end{equation}
which holds for almost all $k\in\RR$, namely those $k \not\in
\frac{1}{3} \ZZ \setminus \ZZ$ for which Eq.~\eqref{eq:h-iter} is
valid.

So, with ${h}(k) = 0$ for almost all $k\in [0,\varepsilon]$, we get
the corresponding property on $[-\varepsilon, 0]$ from ${h} (-k) =
\overline{{h} (k)}$, and then ${h} = 0$ a.e.\ by iteration of
Eq.~\eqref{eq:h-iter-two}, since the additional exception set for the
multiple application of Eq.~\eqref{eq:h-iter-two},
$\bigcup_{n\geqslant 0} \lambda^{-n} \bigl( \frac{1}{3} \ZZ \setminus
\ZZ \bigr)$, is still a null set.
\end{proof}

Similarly, we could consider the matrix $\bigl( h_{ij} (k)\bigr)$ and
its determinant, $\det \bigl( h_{ij} (k)\bigr)$. If the latter
vanishes for a.e.\ $k \in [0,\varepsilon]$ for some $\varepsilon >0$,
the determinant vanishes a.e.\ on $\RR$, because
\[
   \bigl( h_{ij} (\lambda k)\bigr) \, = \;
   \lambda \ts B^{-1} (k) \bigl( h_{ij} (k)\bigr)
   (B^{\dag})^{-1} (k)
\]
holds for a.e.\ $k\in \RR$, which is a full matrix version of
Eq.~\eqref{eq:h-iter-two}. This little observation points the way to a
standard form of $h(k)$ that can be used to simplify the task at hand.

Indeed, Fact~\ref{fact:Hermitian} together with standard arguments
implies that $\bigl( h_{ij} (k)\bigr)$ is a positive semi-definite
Hermitian matrix, for a.e.\ $k\in\RR$. For any admissible $k$, the
matrix is thus of the form $H=\left(\begin{smallmatrix} a & b + \ii c \\
    b - \ii c & d \end{smallmatrix}\right)$ with $a,b,c,d \in \RR$,
$a,d \geqslant 0$ and $a \ts d - (b^2 + c^2) \geqslant 0$. If $a=0$ or
$d=0$, one also has $b=c=0$, the determinant vanishes, and the matrix
has rank at most $1$. In general, whenever the determinant is strictly
positive, $H$ has rank $2$, but can uniquely be decomposed as
\begin{equation}\label{eq:split}
    H \, = \, \begin{pmatrix} a' & b + \ii c \\
    b - \ii c & d \end{pmatrix} + a'' \begin{pmatrix} 1 & 0 \\
    0 & 0 \\ \end{pmatrix}
\end{equation}
with $a = a' + a''$ such that $a' d - (b^2 + c^2) = 0$ and that both
$a'$ and $a''$ are strictly positive. This way, $H$ is split as a sum
of two Hermitian and positive semi-definite matrices of rank $1$.

Now, the inward iteration, in the formulation with $2\!\times\!
2$-matrices, reads
\begin{equation}\label{eq:mat-iter}
   \bigl( h_{ij} \bigl(\tfrac{k}{\lambda}\bigr)\bigr) \, = \;
   \myfrac{1}{\lambda} \, B \bigl( \tfrac{k}{\lambda} \bigr) 
   \bigl( h_{ij} (k)\bigr) B^{\dag} \bigl( 
   \tfrac{k}{\lambda} \bigr)
\end{equation}
and clearly preserves hermiticity and positive semi-definiteness,
while the rank cannot be increased. Moreover, the action is linear, so
respects the splitting of Eq.~\eqref{eq:split}.

\begin{lemma}\label{lem:rank-one}
  For a.e.\ $k \in \RR$, the Radon--Nikodym matrix\/ $\bigl( h_{ij}
  (k)\bigr)$ is Hermitian, positive semi-definite and of rank at
  most\/ $1$.
\end{lemma}

\begin{proof}
  Hermiticity and positive semi-definiteness are a simple consequence
  of Fact~\ref{fact:Hermitian}.  In view of the arguments used in the
  proof of Lemma~\ref{lem:eps-suffice}, it suffices to establish the
  claim on the rank for an interval of the form
  $J=\bigl[ \frac{\varepsilon}{\lambda}, \varepsilon \bigr]$, for some
  $\varepsilon > 0$. To do so, we pick $\varepsilon$ as in
  Proposition~\ref{prop:simple-explode}. If the rank is at most $1$
  for a.e.\ $k\in J$, we are done. If not, there is a subset
  $J'\subseteq J$ of positive measure where the rank of
  $\bigl( h_{ij} (.) \bigr)$ is $2$. Then, for any $k^{}_{0} \in J'$,
  we write
\[ 
    h^{}_{00} (k^{}_{0}) \, = \, a' (k^{}_{0}) + a'' (k^{}_{0})
\]
according to the splitting defined in Eq.~\eqref{eq:split}. In
particular, $a' (k^{}_{0})$ and $a'' (k^{}_{0})$ are then both
strictly positive.

Now, the iteration of the matrix $a'' (k^{}_{0}) \left(
  \begin{smallmatrix} 1 & 0 \\ 0 & 0 \end{smallmatrix} \right)$ under
Eq.~\eqref{eq:mat-iter} corresponds (via a bijective mapping that
preserves the asymptotic growth property) to the iteration of the
vector $a'' (k^{}_{0}) (1,0,0,0)^{t}$ under Eq.~\eqref{eq:h-iter},
which is then (via our change of basis) governed by
Proposition~\ref{prop:simple-explode} and
Corollary~\ref{coro:simple-explode}. So, for a subset $J' \subseteq J$
of positive measure, we get an asymptotic growth as $ c( k^{}_{0})\ts
\frac{1}{k} \, w^{}_{\mathrm{PF}}$ as $k \to 0$. By Lusin's theorem,
there is yet another subset $J''\subset J'$ of positive measure (as
close to that of $J'$ as we want, in fact) such that $a'' (k^{}_{0})$
agrees with a continuous function on $J''$, and this property is
transported to the scaled versions of $J''$ under inward iteration.

However, this is incompatible with $h^{}_{00} (k)$ being locally
integrable: Since $a' (k^{}_{0}) > 0$ for all $k^{}_{0}$ under
consideration, we cannot have any cancellations between
$a' (k)$ and $a'' (k)$, so that $h^{}_{00} (k)$ must grow at least as
$\frac{1}{k}$ for $k \to 0$, for all $k$ along iterations that started
from some $k^{}_{0} \in J''$. This means that rank $2$ on a set of
positive measure is impossible, which establishes the claim.
\end{proof}

\subsection{Dimensional reduction}
Our previous analysis means that we may continue under the assumption
that $\bigl( h_{ij} (k)\bigr)$, for almost every $k\in\RR$, has rank
at most $1$. This is significant due to the following elementary fact
from linear algebra.

\begin{fact}\label{fact:rank-1}
  If\/ $H \in \mathrm{Mat} (2, \CC)$ is Hermitian, positive
  semi-definite and of rank at most\/ $1$, there are two complex
  numbers\/ $v^{}_{0}$ and\/ $v^{}_{1}$ such that\/
  $H_{ij} = {v^{}_{i}}\, \overline{v^{}_{j}}$ holds for\/
  $i,j \in \{ 0,1 \}$. Here, the vectors\/ $(v^{}_{0},v^{}_{1})^{t}$
  and\/ $\ee^{\ii\phi}(v^{}_{0},v^{}_{1})^{t}$ with\/
  $\phi\in[0,2\pi)$ parametrise the same matrix\/ $H$.  \qed
\end{fact}

In Dirac notation, this means $H = | v \rangle \langle v |$, which is
$0$ or a multiple of a projector. Inspecting the iteration
\eqref{eq:mat-iter}, it is then clear that we may consider a vector
$v(k) = \bigl(v^{}_{0} (k), v^{}_{1} (k) \bigr)^{t}$ of functions from
$L^{2}_{\mathrm{loc}}(\RR)$ under the simpler \emph{inward} iteration
\begin{equation}\label{eq:l2iter}
      v \bigl(\tfrac{k}{\lambda} \bigr) \, = \,
    \myfrac{1}{\mbox{\small $\sqrt{\lambda}$}\ts } \, 
    B \bigl( \tfrac{k}{\lambda} \bigr)
    \ts v (k) \ts ,
\end{equation}
which is a considerable dimensional reduction of the iteration problem.  
Likewise, we also have the \emph{outward} analogue
\begin{equation}\label{eq:out-iter}
      v (\lambda k) \, = \, \sqrt{\lambda} \,
      B^{-1} (k)  \ts  \ts v (k) \ts ,
\end{equation}
for $k\not\in Z=\frac{1}{3}\ZZ\setminus \ZZ$. For later use, let us
state a useful property on the asymptotic behaviour of the 
corresponding determinants.

\begin{prop}\label{prop:det-cond}
   For all\/ $k\in\RR$ with\/ $k\not\in\bigcup_{m\geqslant 1}
   \lambda^{m} Z$, one has
\[
    \lim_{n\to\infty} \myfrac{1}{n} \, \log\, \bigl| \det
    \left( B \bigl( \tfrac{k}{\lambda^{n}} \bigr) \cdot \ldots \cdot
    B \bigl( \tfrac{k}{\lambda} \bigr) \right) \bigr|
    \, = \, \log (3) \ts ,
\]
while, for almost all\/ $k\in\RR$ with\/
$k\not\in\bigcup_{m\geqslant 0} \lambda^{-m} Z$, one finds
\[
    \lim_{n\to\infty} \myfrac{1}{n} \, \log \, \bigl| \det
    \bigl( B^{-1} (\lambda^{n-1} k) \cdot \ldots\cdot
    B^{-1} (k) \bigr) \bigr| \, = \, 0 \ts .
\]
\end{prop}

\begin{proof}
  Under the condition on $k$ as stated, which ensures that no matrix
  $B\bigl( \frac{k}{\lambda^{m}}\bigr)$ in the product has determinant
  $0$, the first claim is a simple consequence of $\bigl| \det(M)
  \bigr| = 3$ together with $\lim_{n\to\infty} B\bigl( \frac{k}
  {\lambda^{n}} \bigr) = B(0) = M$, which holds for any fixed $k \in
  \RR$.

  For the second claim, recall that $\det \bigl( B(k)\bigr) = - p(k)$
  with the polynomial $p$ from Lemma~\ref{lem:ida}. When $B(k)$ is
  invertible, we thus get
  $\bigl| \det (B^{-1}(k)) \bigr| = \bigl| 1 + 2\ts \cos (2 \pi k)
  \bigr|^{-1}$,
  which is $1$-periodic in $k$. The invertibility condition excludes
  the set $\bigcup_{m\geqslant 0} \lambda^{-m} Z$, which is a null
  set, from our limit considerations. Otherwise, we get
\begin{equation}\label{eq:birksum}
    \myfrac{1}{n} \, \log \, \bigl| \det
    \bigl( B^{-1} (\lambda^{n-1} k) \cdot \ldots\cdot
    B^{-1} (k) \bigr) \bigr| \, = \, - \myfrac{1}{n}
    \sum_{\ell=0}^{n-1} \log \bigl| \det \bigl(
    B (\lambda^{\ell})\bigr)\bigr| .
\end{equation}

Now, we observe that the sequence $(\lambda^{n} k)^{}_{n\in\NN}$ is
uniformly distributed mod $1$ for almost all $k\in\RR$, which follows
from \cite[Thm.~1.7]{Bugeaud}; see also \cite[Sec.~7.3, Thm.~1]{CFS}
or \cite[Cor.~4.3 and Exc.~4.3]{KN}.  Along such sequences, we thus
sample the real-valued function $\varphi$ defined by
\[
     x \, \longmapsto \, \varphi(x) := - \log \, 
    \bigl| 1 + 2\ts \cos (2 \pi x) \bigr| ,
\]
which also satisfies
$\varphi(x) = - \log \ts \bigl| 1 + z + z^{2} \bigr|$ with
$z=\ee^{2 \pi i x}$. Clearly, $\varphi$ is locally integrable and
$1$-periodic, but not (properly) Riemann integrable. Consequently, we
cannot immediately apply Weyl's uniform distribution result, but need
some intermediate steps.

  The \emph{discrepancy} of the sequence
  $\bigl( k_{n} \bigr)_{n\geqslant 0}$, with $k_n := \lambda^{n} k
  \bmod 1$, is defined as
\[
    \cD^{}_{N} \, = \sup_{0\leqslant a < b \leqslant 1}
    \Bigl| (b-a) - \frac{\card \bigl( [a,b) \cap 
    \{ k^{}_{0}, \ldots, k^{}_{N-1} \}  \bigr)}{N}  \Bigr|
\]
and quantifies the statistical deviation of the first $N$ sequence
elements from (finite) uniform distribution. In our case, for any
fixed $\varepsilon > 0$ and almost all $k\in\RR$, it is given by
\[
     \cD^{}_{N} \, = \, \mathcal{O} \left(\frac{
     \bigl(\log(N)\bigr)^{\frac{3}{2}+\varepsilon}}{\sqrt{N}} \right)
\]
as $N \to \infty$; see \cite[Thm.~5.13]{Harman} or \cite{KHN}.

Now, our condition on $k$ ensures that we never hit one of the
(integrable) singularities along the corresponding sequence, so
$\tri  \lambda^{n} k \tri > 0$ for any such $k$ and all $n\in\NN_{0}$,
where $\tri x \tri$ denotes the distance of $x$ from the nearest
integer.  What is more, again for any fixed $\varepsilon > 0$ and
almost all $k\in\RR$, one has the lower bound
\[
   \min_{0\leqslant n\leqslant N} \tri  \lambda^{n} k
   \tri \, > \, \myfrac{1}{N^{1+\varepsilon}}\ts ,
\]
for all sufficiently large $N$. This follows from a standard argument
on the basis of the Borel--Cantelli lemma; see \cite{Haynes,BHL} for
details.  The same type of bound clearly also applies to the distance
of our sequences from the points of 
$Z=\ZZ + \big\{\frac{1}{3},\frac{2}{3}\big\}$.

Our function $\varphi$ has singularities at $\frac{1}{3}$ and
$\frac{2}{3}$ in the unit interval. If we integrate the derivative
$\varphi^{\ts \prime}$ near such a singularity, starting at distance
$\delta$ say, we will get a contribution of order
$\mathcal{O} \bigl(\log(1/\delta)\bigr)$ from it, as follows from a
simple asymptotic estimate.  Now, with $\delta = 1/N^{1+\varepsilon}$,
the product
\[
   \mathcal{O} \left(\frac{\bigl(\log(N)
   \bigr)^{\frac{3}{2}+\varepsilon}}{\sqrt{N}} \right)\,
   \mathcal{O} \bigl( \log(N^{1+\varepsilon}) \bigr) \, = \,
    \mathcal{O} \left(\frac{\bigl(\log(N)
   \bigr)^{\frac{5}{2}+ 2 \ts\varepsilon}}{\sqrt{N}} \right)
\]
still represents an upper bound that tends to $0$ as
$N\to\infty$. Consequently, by Sobol's theorem from uniform
distribution theory, see \cite[Thm.~1]{Sobol} or \cite[Sec.~2]{Hart}
as well as \cite{BHL},
we may conclude that, for a.e.\ $k \in\RR$, our sampling limit
in Eq.~\eqref{eq:birksum} indeed exists and is given by
\[
     \int_{0}^{1} \varphi (x) \dd x 
     \, =  \, - \int_{0}^{1} \log\, \bigl| 1 + \ee^{2 \pi \ii t}
       + \ee^{4 \pi \ii t} \bigr| \dd t  \, = \, 0 \ts ,
\]
where the integral can be calculated via Jensen's formula from
complex analysis.
\end{proof}

Let us now consider the inward and the outward iteration separately.

\subsection{Properties of inward iteration}

To study the inward iteration more closely, it is instructive to use an
expansion in terms of the eigen\-basis of $B(0)=M$. For fixed $k$, we
thus write the vector $v(k)$ according to Fact~\ref{fact:rank-1} as
\[
    v(k) \, = \, \alpha (k) \begin{pmatrix}\lambda\\ 3\end{pmatrix} + 
                 \beta (k) \begin{pmatrix}1-\lambda\\ 3\end{pmatrix} ,
\]
where the first vector is proportional to $v^{}_{\mathrm{PF}}$ from
Eq.~\eqref{eq:lettfreq}, while the second vector belongs to the
eigenvalue $1-\lambda$. An explicit calculation of this change of
basis now shows that iteration \eqref{eq:l2iter} is equivalent to
\begin{equation}\label{eq:newiter}
 \begin{pmatrix} \alpha \bigl(\tfrac{k}{\lambda}\bigr)\\[1mm] 
  \beta \bigl(\tfrac{k}{\lambda}\bigr)\end{pmatrix} \, = \,
  \myfrac{1}{\mbox{\small $\sqrt{\lambda}$}} 
  \begin{pmatrix}\lambda\, \alpha (k)\\[1mm] 
  (1-\lambda)\, \beta (k)\end{pmatrix} - 
  \frac{z\bigl(\tfrac{k}{\lambda}\bigr)}
    {\mbox{\small $\sqrt{\lambda}$}}
  \, N  \begin{pmatrix} \alpha (k)\\[1mm] \beta (k)\end{pmatrix}
  \, = \,  \myfrac{1}{\mbox{\small $\sqrt{\lambda}$}} \Bigl( D - 
  z\bigl(\tfrac{k}{\lambda}\bigr) \, N \Bigr)
   \begin{pmatrix} \alpha (k)\\[1mm] \beta (k)\end{pmatrix} ,
\end{equation}
with $z(k)=3-p(k)$, where $p$ is the trigonometric polynomial from
Lemma~\ref{lem:ida}, the diagonal matrix $D = \mathrm{diag} (\lambda,
1-\lambda)$, which is the diagonalisation of $M$, and the constant
matrix
\[
   N \, = \, \myfrac{1}{39} \begin{pmatrix}
   -3+6\lambda & 10-7\lambda \\
   3+7\lambda & 3-6\lambda  \end{pmatrix} .
\]
Here, $N$ is nilpotent with $N^{2}=0$ and kernel $\CC (6\lambda - 3, 7
\lambda + 3)^{t}$.  The result of this property is that, if $(\alpha
(k), \beta(k) )^{t}$ is in this kernel, the next iterate under
\eqref{eq:newiter} has equal components.  Note that $z(k)=-6\pi\ii\ts
(1+\lambda)k+\mathcal{O}(k^2)$ as $k\to 0$, and that $N$ has matrix
norms
\[
   \|N\|^{}_{1} \, = \, \|N\|^{}_{\infty}
   \, = \, \myfrac{\lambda}{3} \, \approx \, 0.768
   \quad \text{and} \quad \|N\|^{}_{2} \, = \,
   \frac{7 \ts (2\lambda-1)}{39} \, \approx \, 0.647 \ts .
\]

At this point, with $Z=\frac{1}{3} \ZZ \setminus \ZZ$, we can state
the general structure of the inward iteration as follows, and refer to
Appendix A for the details of the underlying calculations and
estimates.

\begin{prop}\label{prop:inward}
  For any\/ $k\in\RR$ with\/
  $k\not\in\bigcup_{m\geqslant 1} \lambda^{m} Z$, the inward iteration
  of Eq.~\eqref{eq:l2iter}, or equivalently that of
  Eq.~\eqref{eq:newiter}, is Lyapunov regular, with the two Lyapunov
  exponents
\[
     \chi^{(1)}_{-} \, = \, \log \myfrac{\lambda - 1}{\sqrt{\lambda}}
     \; < \; 0 \; < \; \chi^{(2)}_{-} \, = \, \log \sqrt{\lambda}\ts ,    
\]   
which are independent of\/ $k$. At each such\/ $k$, we have a matching
vector space filtration
\[
      \{ 0 \}  =  E^{(0)}_{-} (k) \, \subsetneq \,
      E^{(1)}_{-} (k) \, \subsetneq \,
      E^{(2)}_{-} (k)  = \CC^{2}
\]  
with the equivariance condition\/ $E^{(1)}_{-} \bigl(
\frac{k}{\lambda}\bigr) = B\bigl( \frac{k}{\lambda}\bigr) \,
E^{(1)}_{-} (k)$ for all admissible\/ $k\in\RR$.
\end{prop}

\begin{proof}
The equivalence of the two iterations, which emerge from one another by 
a simple change of basis, for Lyapunov theory follows from standard 
arguments; compare \cite[Sec.~1.2]{BP1}.

For any fixed $0 < k \leqslant \varepsilon$, with sufficiently small
$\varepsilon$ as detailed in Appendix A, the existence of the second
exponent follows from a compactness argument as explained before
Eq.~\eqref{eq:in-lower}. This also implies the filtration as stated;
see \cite{BP2} for a formulation of Lyapunov theory with complex
matrices as needed here.

For $k > \varepsilon$, we end up in our smaller interval after
finitely many iterations.  Provided no matrix in the iteration has
vanishing determinant, which is the reason for the extra condition in
the statement, the problem is thus reduced to the previous case, and
the filtration is transported back by a simple matrix inversion, which
also implies the claimed equivariance condition.  The situation for
negative $k$ maps to that for positive $k$ via complex conjugation,
and is thus completely analogous.

The exponents are now as stated, which follows from
Eqs.~\eqref{eq:in-upper} and \eqref{eq:in-lower}, and sum up to $\log
(\lambda - 1)$, which agrees with a specific limit,
\[
     \lim_{n\to\infty} \,\myfrac{1}{n} \ts \log \, \Bigl|
     \det \Bigl(\tfrac{1}{\sqrt{\lambda}\ts} B 
     \bigl(\tfrac{k}{\lambda^{n}}\bigr)\Bigr)
     \cdot \ldots \cdot \det \Bigl(\tfrac{1}{\sqrt{\lambda}\ts} 
     B\bigl( \tfrac{k}{\lambda}
     \bigr)\Bigr)\Bigr| \, = \, \log (3) - \log (\lambda) 
     \, = \, \log (\lambda - 1)\ts ,
\]
by an application of Proposition~\ref{prop:det-cond} for any of the
admissible values of $k$. This establishes the claimed regularity;
compare \cite[Sec.~1.3.2]{BP1}.
\end{proof}

\begin{remark}
  Let us note that the result of Proposition~\ref{prop:inward} can
  also be seen as a consequence of
  $\lim_{n\to\infty} B \bigl( \frac{k}{\lambda^{n}}\bigr) = B(0) = M$,
  which holds for all $k\in\RR$. Indeed, the resulting Lyapunov
  exponents of the matrix cocycle
  $B \bigl( \frac{k}{\lambda^{n}} \bigr) \cdot \ldots \cdot B \bigl(
  \frac{k}{\lambda} \bigr)$
  equal those of the iteration of $M$ alone, provided the overall
  determinant never vanishes. For all admissible $k$, which are those
  stated in Proposition~\ref{prop:inward}, the Lyapunov spectrum thus
  is $\{ \log (\lambda), \log (\lambda - 1)\}$, which implies the
  above result for the scaled iteration with
  $\frac{1}{\sqrt{\lambda}} B(k)$.  \exend
\end{remark}

Unfortunately, the determination of the Lyapunov structure for the
inward iteration does not yet suffice to rule out an absolutely continuous
component, because one of the relevant exponents is negative. To exclude 
the corresponding solution, one has to consider its behaviour in the
outward direction, where it will be unbounded. We thus turn
to a general investigation of the complementary iteration direction.

\subsection{Properties of outward iteration}

Here, we look at the asymptotic behaviour of the outward iteration,
first without the extra factor $\sqrt{\lambda}$ that is present in
Eq.~\eqref{eq:out-iter}. We shall indicate a reference to this version
by a tilde on the exponents, while we omit it when we speak of the
exponents for Eq.~\eqref{eq:out-iter} \emph{including} the extra
factor.  The pointwise extremal Lyapunov exponents, compare 
\cite[Ch.~3]{Viana} --- provided they exists as limits --- would now be
\begin{align*}
   \tilde{\chi}^{(1)}_{+} (k) \, & = - \lim_{n\to\infty} \myfrac{1}{n}
   \log \big\| B( k) \cdot \ldots \cdot
   B (\lambda^{n-1} k) \big\|  
\intertext{and} 
   \tilde{\chi}^{(2)}_{+} (k) \, & = \lim_{n\to\infty} \myfrac{1}{n}
   \log \big\| B^{-1} (\lambda^{n-1} k) \cdot \ldots \cdot
    B^{-1} (k)  \big\| ,
\end{align*}
together with
$\tilde{\chi}^{(1)}_{+}(k) \leqslant \tilde{\chi}^{(2)}_{+}
(k)$. Since $B^{-1} = \frac{1}{\det (B)} B^{\mathsf{ad}}$ for any 
invertible matrix $B$, with $B^{\mathsf{ad}}$ denoting the adjoint
matrix, the right hand side for $\tilde{\chi}^{(2)}_{+}$ can 
alternatively be written as
\[
    \tilde{\chi}^{(2)}_{+} (k) \,  = \lim_{n\to\infty} 
    \myfrac{1}{n} \log \big\| B^{\mathsf{ad}} (\lambda^{n-1} k) 
    \cdot \ldots \cdot B^{\mathsf{ad}} (k) \big\| \,
    - \lim_{n\to\infty} \myfrac{1}{n} \sum_{\ell=0}^{n-1}
    \log \, \bigl| \det (B (\lambda^{\ell} k)) \bigr|
\]
where we already know from Proposition~\ref{prop:det-cond} that the
second limit, for a.e.\ $k\in\RR$, exists and equals $0$. Note that
$\det(B(k))$ has zeros, but that $B(k)=0$ is impossible, wherefore
$\| B(k)\|>0$ holds for all $k\in\RR$.

To establish the existence of the limits, we would need some version
of Kingman's subadditive ergodic theorem and later its consequence in
the form of an Oseledec-type multiplicative ergodic theorem. The
difficulty here is that we are dealing with an infinite measure space
and with a non-stationary sequence of matrices. As was shown in
\cite{FSS}, this is still possible in some cases via the structure of
almost periodic functions with joint almost periods. Unfortunately, as
far as we are aware, this would still require our eigenvalue $\lambda$
to be a PV number, which it is not. We therefore change our perspective
by defining the extremal exponents as
\begin{equation}\label{eq:limsup}
\begin{split}
 \tilde{\chi}^{(1)}_{+} (k) \, & := - \limsup_{n\to\infty} \myfrac{1}{n}
   \log \big\| B( k) \cdot \ldots \cdot
   B (\lambda^{n-1} k) \big\|  \, \quad \text{and} \\ 
   \tilde{\chi}^{(2)}_{+} (k) \, & := \limsup_{n\to\infty} \myfrac{1}{n}
   \log \big\| B^{-1} (\lambda^{n-1} k) \cdot \ldots \cdot
    B^{-1} (k)  \big\| ,
\end{split}
\end{equation}
which always exist. As we shall see, this will still give us
useful bounds that are strong enough for our purposes.

Clearly, it does not matter which matrix norm we use to define the
exponents, because all norms on $\mathrm{Mat} (2, \CC)$ are equivalent
and convergence in one norm implies convergence in any other norm,
with the same limit due to the logarithm involved. Now, in the
Frobenius norm, we have
$\| A \|^{}_{\mathrm{F}} = \| A^{\mathsf{ad}} \|^{}_{\mathrm{F}}$ for
any $2\!\times\!2$-matrix. Since
$(A B)^{\mathsf{ad}} = B^{\mathsf{ad}} A^{\mathsf{ad}}$, the following
observation is immediate.

\begin{lemma}\label{lem:extremal}
  Let\/ $k\in\RR$ be any element of the set of numbers of full measure
  for which the second limit of Proposition~\emph{\ref{prop:det-cond}}
  holds. Then, when using the definition from Eq.~\eqref{eq:limsup},
  the extremal Lyapunov exponents satisfy the relation\/
  $\tilde{\chi}^{(1)}_{+} (k) + \tilde{\chi}^{(2)}_{+} (k) = 0$.
  
  In addition, whenever\/
  $\frac{1}{n} \log \ts \| B (k) \cdot \ldots \cdot B (\lambda^{n-1}
  k) \|$
  converges as\/ $n \to\infty$, both extremal Lyapunov exponents at
  this\/ $k$ exist as limits as well.  \qed
\end{lemma}

To continue a little in this direction, let us state a result that can
replace the usual argument with an invariant measure for the
transformation $k \mapsto \lambda k$.

\begin{lemma}\label{lem:constant}
  Let\/ $k\in\RR$ such that\/ $\det( B(\lambda^{m} k) ) \ne 0$ holds
  for all\/ $m\in\NN_{0}$, which only excludes a countable set. Then,
  the exponent\/ $\tilde{\chi}^{(1)}_{+} (k)$ exists as a limit if and
  only if\/ $\tilde{\chi}^{(1)}_{+} (\lambda k)$ does, and the two
  values agree.  In this case, one has\/
  $\tilde{\chi}^{(1)}_{+} (k) = \tilde{\chi}^{(1)}_{+} (\lambda^{m}
  k)$
  for all\/ $m\in\NN$, which is to say that the exponents exist and
  are constant along the sequence\/
  $(\lambda^{m} k)^{}_{m\in\NN_{0}}$.
\end{lemma}

\begin{proof}
  The statement is trivial for $k=0$. Since $B(-k) = \overline{B(k)}$,
  it suffices to consider $k>0$.  Let us use the abbreviation
  $B^{(n)} (k) = B(k) \cdot \ldots \cdot B(\lambda^{n-1} k)$, whence
  one has the recursion relation
\begin{equation}\label{eq:B-rec}
   B^{(m+n)} (k) \, = \, B^{(m)} (k) \ts
   B^{(n)} (\lambda^{m} k)
\end{equation}
for any $m,n\in\NN$.  Observing $\| A B \| \leqslant \| A \| \| B \|$
as well as
$\| B \| = \| A^{-1} A B\| \leqslant \| A^{-1} \| \| A B \|$, hence
$\| A B \| \geqslant \| B \| / \| A^{-1} \|$, for any invertible $A$,
one can derive the estimates
\[
 \frac{ \log \big\| B^{(n+\ell)} (k) \big\|}{n} -
    \frac{\log \big\| B^{(\ell)}\bigr (k) \big\|}{n}
     \leqslant \,
    \frac{\log  \big\| B^{(n)} (\lambda^{\ell} k) \big\|}{n}
    \, \leqslant \,  \frac{ \log \big\| B^{(n+\ell)}
    (k) \big\|}{n} +
    \frac{\log \big\| \bigl(B^{(\ell)}\bigr)^{-1} (k) \big\|}{n} 
\]
and
\[
    \frac{ \log \big\| B^{(n)} (\lambda^{\ell} k) \big\|}{n+\ell} -
    \frac{\log \big\| \bigl(B^{(\ell)}\bigr)^{-1} (k) \big\|}{n+\ell} 
     \leqslant \,
    \frac{\log  \big\| B^{(n+\ell)} (k) \big\|}{n+\ell}
    \, \leqslant \,  \frac{ \log \big\| B^{(n)}
    (\lambda^{\ell} k) \big\|}{n+\ell} +
    \frac{\log \big\| B^{(\ell)} (k) \big\|}{n+\ell} 
\]
for any fixed $\ell\in\NN$. The claims now follow from standard
arguments.
\end{proof}

Let $p$ be the trigonometric
polynomial from Lemma~\ref{lem:ida}. Now, define
trigonometric polynomials $P_{n}$ by $P_{-1}\equiv 0$ and
$P_{0} \equiv 1$ together with the recursion
\begin{equation}\label{eq:P-rec}
   P_{n+1} (k) \, = \, P_{n} (k) + p(\lambda^{n} k) \,
   P_{n-1} (k)
\end{equation}
for $n\geqslant 0$, which gives $P_{1}\equiv 1$,
$P_{2} (k) = 1 + p (\lambda k)$ and so on.  These 
trigonometric polynomials
are related with the matrix function $B^{(n)} (k)$ from the proof of
Lemma~\ref{lem:constant} as follows.

\begin{fact}\label{fact:P-rec}
  For all\/ $n\in\NN$, one has
\[
    B^{(n)} (k) \, = \, \begin{pmatrix}
    P_{n} (k) & P_{n-1} (k) \\
    \, p(k) \, P_{n-1} (\lambda k) &
    p(k) \, P_{n-2} (\lambda k) \end{pmatrix},
\]
with the trigonometric polynomials from Eq.~\eqref{eq:P-rec},
together with
\[
  \big\| B^{(n)} (k) \big\|^{2}_{\mathrm{F}} \, = \,
  \lvert P_{n} (k) \rvert^{2} + 
  \lvert P_{n-1} (k) \rvert^{2}
  + \lvert p(k) \rvert^{2} \bigl(
  \lvert P_{n-1} (\lambda k) \rvert^{2} + 
  \lvert P_{n-2} (\lambda k) \rvert^{2} \bigr)
\]
for the squared Frobenius norm. Moreover, for all\/ $n\geqslant 0$
and\/ $k\in\RR$, one also has the recursion
\[
     P_{n+1} \bigl(\tfrac{k}{\lambda}\bigr) \, = \, 
     P_n (k) + p (k) \,  P_{n-1} (\lambda k) \ts .
\]
\end{fact}

\begin{proof}
The first claim follows by induction, where one can employ
Eq.~\eqref{eq:B-rec} in the form $B^{(n+1)} (k) = 
B^{(n)} (k) \, B (\lambda^n k)$ for $n\geqslant 1$. 
The second formula is then an immediate
consequence. Finally, the alternative recursion follows
inductively via a different use of Eq.~\eqref{eq:B-rec}, 
this time giving $B^{(n+1)} (k) = B (k) \, B^{(n)} (\lambda k)$.
Comparing with the first formula, and replacing $k$ by
$\frac{k}{\lambda}$, leads to the alternative recursion.
\end{proof}

For any $n\in\NN$, the mapping
$k \mapsto \| B^{(n)} (k) \|^{2}_{\mathrm{F}}$ defines a non-negative
trigonometric polynomial by Fact~\ref{fact:P-rec}.  Moreover, as a
consequence of Fact~\ref{fact:quasi}, it is a \emph{quasiperiodic}
function with two fundamental frequencies, $1$ and $\lambda$. This is
so because all higher powers of $\lambda$ can be written as an integer
linear combination of $1$ and $\lambda$ due to the relation
$\lambda^2 = \lambda + 3$.  In particular, one has
\begin{equation}\label{eq:lam-powers}
   \lambda^n \, = \, a^{}_{n} \lambda + b^{}_{n}
   \quad \text{with} \quad
   \binom{a^{}_{n}}{b^{}_{n}}
   \, = \, M^n  \binom{0}{1}
\end{equation}
where $M$ is the substitution matrix from Eq.~\eqref{eq:submat}. Note
that this holds for all $n\in\ZZ$, with $a^{}_{0} = 0$ and
$b^{}_{0} = 1$.  Since $\lambda$ is not a unit, only the coefficients
with non-negative index are integers, while the other ones are
rational numbers. Let us note some further properties.

\begin{fact}\label{fact:gcd}
  For all\/ $n\in\NN$, the coefficients defined by
  Eq.~\eqref{eq:lam-powers} satisfy 
\begin{enumerate}\itemsep=2pt
\item $a^{}_{n}\equiv 1 \bmod{3}$ and\/ $b^{}_{n}\equiv 0 \bmod{3}$, 
   as well as
\item $\gcd(a^{}_{n},a^{}_{n+1})=1$ and\/ $\gcd(b^{}_{n},b^{}_{n+1})=3$.
\end{enumerate}
\end{fact}
\begin{proof}
  Observe first from Eq.~\eqref{eq:lam-powers} that $b^{}_{n}=3\ts
  a^{}_{n-1}$ and $a^{}_{n+1}=a^{}_{n}+3\ts a^{}_{n-1}$ for all
  $n\in\NN$.  Clearly, the claims on the $b^{}_{n}$ thus follow from
  those on the $a^{}_{n}$.  Since $a^{}_{0}=0$ and
  $a^{}_{1}=b^{}_{0}=1$, the congruence relation for the $a^{}_{n}$ is
  clear by induction.

  Since also $a^{}_{2}=1$, we have $\gcd(a^{}_{1},a^{}_{2})=1$. For
  $n\in\NN$, the recursion gives
\[ 
    \gcd(a^{}_{n},a^{}_{n+1}) \, = \, \gcd(a^{}_{n},a^{}_{n}+3\ts a^{}_{n-1}) 
    \, = \,  \gcd(a^{}_{n},3\ts a^{}_{n-1}) \, = \, \gcd(a^{}_{n},a^{}_{n-1}),
\]
where the last step follows from the congruence property previously
established. The claim is now clear by induction.
\end{proof}

Before we continue, let us state a useful property. Recall that a
function $f $ from $C_{\mathsf{u}} (\RR)$, the space of uniformly
continuous and bounded functions on $\RR$, is Bohr (or uniformly)
\emph{almost periodic} if the set of $\varepsilon$-almost periods
$P^{\infty}_{\varepsilon} := \{ t \in \RR : \| f - T^{}_{t} f
\|^{}_{\infty} < \varepsilon\}$
is relatively dense in $\RR$ for every $\varepsilon>0$. 
Here, $\|.\|^{}_{\infty}$ denotes the supremum norm on
$C_{\mathsf{u}} (\RR)$, and $T^{}_{t} f$ is defined by
$\bigl(T^{}_{t} f \bigr) (x) = f(x-t)$; see \cite{Cord} for
background.  

\begin{fact}\label{fact:log-fp}
  Let\/ $f$ be a real-valued Bohr almost periodic function such
  that\/ $f(x) \geqslant a > 0$ for all\/ $x\in\RR$ and some fixed\/
  $a$. Then, $\log (f)$ is Bohr almost periodic as well. 
\end{fact}

\begin{proof}
  Since $f$ is Bohr almost periodic, it is bounded, so
  $f(x) \in [a,b]$ for some $b\geqslant a >0$ by assumption. Now, the
  logarithm is uniformly continuous on $[a,b]$, which means that, for
  any $\varepsilon>0$, there is a
  ${\delta} = {\delta} (\varepsilon) > 0$ such that
  $\lvert \ts \log (x) -\log (y) \rvert < {\varepsilon}$ whenever
  $\lvert x-y \rvert < \delta$.

  Let $\varepsilon>0$ be arbitrary and let
  $\delta = \delta(\varepsilon)$.  Now, let $t$ be any of the
  relatively dense $\delta$-almost periods of $f$, hence
  $\lvert f(x) - f(x-t) \rvert < \delta$ for all $x\in\RR$. Then, $t$
  also is an $\varepsilon$-almost period of $\log(f)$, which
   implies the claim.
\end{proof}

\begin{lemma}\label{lem:away}
  For any\/ $n\in\NN$, there is a constant\/ $\delta_n >0$ such 
  that\/ $\| B^{(n)} (k)\|^{2}_{\mathrm{F}} \geqslant \delta_n $ 
  holds for all\/ $k\in\RR$. Consequently, also\/ 
  $k \mapsto \log \| B^{(n)} (k) 
  \|^{2}_{\mathrm{F}}$ defines a Bohr almost periodic function.
\end{lemma}

\begin{proof}
  The second claim follows from the first by Fact~\ref{fact:log-fp}.
  To show the first claim, observe that, with $0 \leqslant \lvert p(k)
  \rvert \leqslant 3$, we have
\[
   \delta^{}_{1} \, = \, 2 \, \leqslant \,
   2 + \lvert p(k) \rvert^{2} \, = \, 
   \| B (k) \|^{2}_{\mathrm{F}} \, \leqslant \,11 \ts .
\]
In particular, the rank of $B^{(1)} (k) = B(k)$ is at least $1$, and
it is $2$ whenever $\det(B(k)) \ne 0$. Similarly, one finds
\[
  \delta^{}_{2} \, = \, 1 \, \leqslant \,
  1 + 2 \ts \lvert p(k) \rvert^{2} +
  \lvert 1 + p(\lambda k)\rvert^{2} \, = \,
  \| B^{(2)} (k) \|^{2}_{\mathrm{F}} \, \leqslant \, 35 \ts ,
\]
while the lower bounds become considerably more involved 
after this. Still, we can see inductively
that $\| B^{(n)} (k) \|^{2}_{\mathrm{F}} > 0$ for all $n\in\NN$ and
all $k\in\RR$ as follows. 

Observe that $\det( B^{(n)} (k))=0$ if and only if $k\in
\bigcup_{\ell=0}^{n-1} \lambda^{-\ell} Z$ with $Z= \ZZ +
\bigl\{ \frac{1}{3}, \frac{2}{3} \bigr\}$ as before. Then, the
zero set of $\det(B(k))$, which is $Z$, and that of
$\det(B^{(n)} (\lambda k))$, which is $\bigcup_{\ell=1}^{n}
\lambda^{-\ell} Z$, are disjoint. Now, the rank of
$B^{(n+1)} (k) = B(k) \, B^{(n)} (\lambda k)$ is $2$ unless
$\det( B^{(n+1)} (k)) = 0$. In the latter case, we either
have $k\in Z$, where $B(k)$ has rank $1$ while 
$B^{(n)} (\lambda k)$
is of full rank, or $k\in \bigcup_{\ell=1}^{n} \lambda^{-\ell}
Z$, where $B(k)$ has full rank and $B^{(n)} (\lambda k)$ has
rank $1$ by induction. In these two cases, $B^{(n+1)} (k)$
still has rank $1$ then, so cannot be the $0\ts $-matrix.
Altogether, this implies $\| B^{(n+1)} (k) \|^{2}_{\mathrm{F}}
>0$ for all $k\in\RR$.

The harder part to show is that, for fixed $n$,
$\| B^{(n)} (k) \|^{2}_{\mathrm{F}}$ is bounded away from $0$. As this
function is quasiperiodic with fundamental frequencies $1$ and
$\lambda$ as a consequence of Fact~\ref{fact:quasi}, it can be written
as
\[
   \| B^{(n)} (k) \|^{2}_{\mathrm{F}} \, = \,
   F_{n} (x,y) \big|_{x=\lambda k, \, y= k} 
\]
with $F_n \geqslant 0$ a smooth, doubly $1$-periodic function 
on $\RR^2$.  Now, our claim is equivalent to
$F_{n} (x,y) \geqslant \delta_{n}>0$ on the compact set $[0,1]^2$.
In fact, since $F_n$ is continuous, the latter property follows if we 
show that $F_n$ has no zero in $[0,1]^2$, as we then get
\[
    \delta_n \, := \, \min \ts 
     \{ F_n (x,y) : 0 \leqslant x,y \leqslant 1 \} 
     \, > \, 0 \ts .
\]

By Fact~\ref{fact:quasi}, we have a representation $B^{(n)} (k) =
\tilde{B}^{(n)} (x,y) \big|_{x=\lambda k, \, y= k}$ with
\begin{equation}\label{eq:xy-trick}
    \tilde{B}^{(n+1)} (x,y) \, = \, \tilde{B} (x,y) \,
    \tilde{B}^{(n)} (x\nts + \nts 3y, x) \, = \,  \tilde{B}^{(n)} (x,y) \,
    \tilde{B} (a^{}_{n+1} x + b^{}_{n+1} y, a^{}_{n} x + b^{}_{n} y)
\end{equation}
for $n\geqslant 0$ and the coefficients from
Eq.~\eqref{eq:lam-powers}.  In particular, if $\tilde{p}$ is the
function from Fact~\ref{fact:quasi} and $\tilde{P}_n$ is the
corresponding `lift' of $P_n$, Eq.~\eqref{eq:P-rec} and
Fact~\ref{fact:P-rec} turn into the analogous relations for
$\tilde{p}$, $\tilde{P}_{n}$ and $\tilde{B}^{(n)}$, where our
recursion now reads
\begin{equation}\label{eq:tilde}
\tilde{P}^{}_{n+1} (x,y) \, = \, \tilde{P}^{}_n (x,y) +
\tilde{p}(a^{}_{n+1} x + b^{}_{n+1} y, a^{}_{n} x + b^{}_{n} y ) \,
\tilde{P}^{}_{n-1} (x,y)
\end{equation}
for $n\in\NN$.  If we show that our previous argument generalises to
give $\tilde{B}^{(n)} (x,y) \ne 0$ for all $n\in\NN$ and all
$x,y \in [0,1]$, we are done.

Since this property is clearly true for $n=1$, assume that it holds
for some fixed $n\geqslant 1$. We see from the first identity in
Eq.~\eqref{eq:xy-trick} that $\tilde{B}^{(n+1)} (x,y)$ can only vanish
when both factors have rank less than $2$, which means (by our
induction hypothesis and the structure of $\tilde{B}$) that both must
have rank $1$ because none can be the $0$-matrix. This implies
$p(x,y)=0$ and hence $y \in Z$, whence we get
\[ 
     \tilde{B}^{(n+1)} (x,y)  \big|_{y\in Z}\, = \, 
     \begin{pmatrix} 1 & 1 \\ 0 & 0 \end{pmatrix}  
     \, B^{(n)} (x+3y,x) \big|_{y\in Z} \, = \,
     \begin{pmatrix} \tilde{P}_{n+1} (x,y)\big|_{y\in Z} 
     & \tilde{P}_n (x,y)\big|_{y\in Z} \\  0 & 0  \end{pmatrix}.
\]
This can only be the $0$-matrix if $\tilde{P}_n$ and $\tilde{P}_{n+1}$
have a common zero subject to the constraint $y \in Z$.
We will now show that this is impossible.  

Since all $\tilde{P}_{m}$ are $1$-periodic in both arguments, the
condition $y\in Z$ means that we only need to consider the values of
$\tilde{P}_{n}$ and $\tilde{P}_{n+1}$ on the line segments
$\bigl\{0\leqslant x\leqslant 1 , \, y=\frac{1}{3} \bigr\}$ and
$\bigl\{0\leqslant x\leqslant 1 ,  \, y=\frac{2}{3} \bigr\}$.  We clearly
have $\overline{\tilde{p}(x,y)} = p(-x,-y)$ for all $x,y\in\RR$, which
implies
\[
    \overline{\tilde{P}_{m}(x,y)} \, = \, 
    \tilde{P}_{m}(-x,-y) \, = \, 
    \tilde{P}_{m}(1-x,1-y)
\]
for $m \in \{ n,n+1\}$.  Consequently, it suffices to consider
$y=\frac{1}{3}$. With $q=\ee^{2\pi\ii x}$, one finds
\[
    \tilde{P}_{m}\bigl(x,\tfrac{1}{3} \bigr) \, = \, 
    1+q+\ldots + q^{a^{}_{m+1}-1}
\]
for $m\geqslant 0$, which can be shown by induction from
Eq.~\eqref{eq:tilde} together with 
\[
  \tilde{p}(a^{}_{m+1} x + b^{}_{m+1} y, 
  a^{}_{m} x + b^{}_{m} y)\big|_{y=\frac{1}{3}} \,
  = \, \tilde{p}(a^{}_{m+1} x , a^{}_{m} x ) \, = \,
  q^{a^{}_{m+1}} \bigl(1+q^{a^{}_{m}} + q^{2a^{}_{m}}\bigr),
\]
where we have used Fact~\ref{fact:gcd}\ts(1). Clearly,
$\tilde{P}_{m}\bigl(x, \frac{1}{3} \bigr)$ vanishes if and only if
$q^{a^{}_{m+1}}=1$ with $q\neq 1$. By Fact~\ref{fact:gcd}\ts(2), the
integers $a^{}_{n}$ and $a^{}_{n+1}$ are coprime which means that the
zero sets of $\tilde{P}_{n}$ and of $\tilde{P}_{n+1}$ along the line
$\bigl\{y=\frac{1}{3}\bigr\}$ are disjoint, and our
argument is complete.
\end{proof}

Let us now define $L_n (k) = \log \| B^{(n)} (k) \|$, where
$\|.\|$ is any (fixed) matrix norm.  Clearly, any $L_n$ is 
a quasiperiodic function.
Let us extend this by setting $L^{}_{0} \equiv 0$.
For arbitrary $m,n\in\NN$, we now have the subadditivity relation
\begin{equation}\label{eq:L-subadd}
    L_{m+n} (k) \, \leqslant \, L_{m} (k) +
    L_{n} (\lambda^{m} k) \ts ,
\end{equation}
which holds as a consequence of
Eq.~\eqref{eq:B-rec} and our definition of $L^{}_{0}$. Note that the
function $L_n$, for $n\in\NN$ and  all $k\in\RR$, is bounded as
\begin{equation}\label{eq:bounded}
  -\infty \, < \, \log(\delta_{n}) \, \leqslant \,
  L_{n} (k) \, \leqslant \, n \max_{0\leqslant x,y \leqslant 1}
  \log \| \tilde{B} (x,y) \| \, < \, \infty 
\end{equation}
due to Lemma~\ref{lem:away} in conjunction with the above
subadditivity and Fact~\ref{fact:quasi}.

To continue, we need the \emph{mean} of a function $f$, which is
defined as
\[
   \MM (f) \, := \, \lim_{T\to\infty} \myfrac{1}{T} \!
    \int_{0}^{T} \!\! f(t) \dd t
\]
and exists for all weakly almost periodic functions \cite{Cord},
which certainly include continuous, quasiperiodic functions. In
particular, we have
\[
   \MM (L_n ) \, = \, \lim_{T\to\infty} \myfrac{1}{T} \!
    \int_{0}^{T} \!\! L_{n} (t) \dd t \, = \int_{[0,1]^{2}}
    \log \| \tilde{B}^{(n)} (x,y) \| \dd x \dd y \ts .
\]
The crucial connection is now the following.

\begin{lemma}\label{lem:subadd}
  For any fixed\/ $N\in\NN$ and a.e.\ $k\in\RR$, one has
\[
   \limsup_{n\to\infty} \myfrac{1}{n} \ts L^{}_{n} (k)
   \, \leqslant \, \myfrac{1}{N} \ts \MM (L^{}_{N}) \ts .
\]
\end{lemma}

\begin{proof}
  Fix $N\in\NN$ and observe that any $n\in\NN$ has a unique
  representation as $n = m N + r$ with $m\in\NN_{0}$ and
  $0\leqslant r < N$. By Eq.~\ref{eq:L-subadd}, we get
\[
   L^{}_{n} (k) \, = \, L^{}_{mN+r} (k) \, \leqslant 
   L^{}_{r} (\lambda^{mN} k) \, + 
   \sum_{\ell=0}^{m-1} L^{}_{N} (\lambda^{\ell N} k) \ts .
\]
With $L^{}_{0}= 0$, the functions $L_{r}$ with 
$0\leqslant r < N$ are uniformly
bounded from above and below as a consequence of
Eq.~\eqref{eq:bounded}. So, we know that
$\lim_{n\to\infty} \frac{1}{n} L^{}_{r} (\lambda^{mN} k) =0$. 
Next, we have
\[
   \myfrac{1}{n} \sum_{\ell=0}^{m-1} L^{}_{N} (\lambda^{\ell N} k)
   \, = \, \myfrac{mN}{mN+r} \, \myfrac{1}{N} \biggl(
   \myfrac{1}{m} \sum_{\ell=0}^{m-1} L^{}_{N} (\lambda^{\ell N} k)
   \biggr) \, \xrightarrow{\, m \to\infty\,} \,
   \myfrac{1}{N} \, \MM (L^{}_{N} ) \ts ,
\]
where the last step holds for a.e.\ $k\in\RR$ as a consequence of the
uniform distribution property (modulo $1$) of the sequence
$(\lambda^{\ell N}k)^{}_{\ell\in\NN}$ for a.e.\ $k\in\RR$ in
conjunction with the quasiperiodicity of $L^{}_{N}$; see
\cite[Lemma~2.2]{FSS}, or \cite[Thm.~6.4.4]{BHL} for a detailed
derivation.

 Our claim is now a simple consequence.
\end{proof}

At this point, we can formulate our result as follows.

\begin{prop}\label{prop:L-out}
  For a.e.\ $k\in\RR$, the extremal Lyapunov exponents for the outward
  iteration, defined as in Eq.~\eqref{eq:limsup}, satisfy\/
  $\tilde{\chi}^{(1)}_{+} (k) + \tilde{\chi}^{(2)}_{+} (k) = 0$.

  Moreover, for a.e.\ $k\in\RR$, their absolute value is bounded by
\[
    \tilde{D} \, := \, \inf_{n\in\NN} \ts \myfrac{1}{n} \ts \MM \bigl(
    \log \| B^{(n)}(.) \| \bigr) \,
     = \, \inf_{n\in\NN} \myfrac{1}{2 n} \int_{[0,1]^2}
    \log \bigl( \| \tilde{B}^{(n)} (x,y) \|^{2}_{\mathrm{F}} 
     \bigr) \dd x \dd y \ts .
\]
\end{prop}

\begin{proof}
  The first claim follows from Lemma~\ref{lem:extremal}. The expression
  of the bound $\tilde{D}$ for $\lvert \tilde{\chi}^{(1)}_{+} (k) \rvert$
  and $\lvert \tilde{\chi}^{(2)}_{+} (k) \rvert$ as an infimum is an obvious 
  consequence of Lemma~\ref{lem:subadd}. This infimum can be
  calculated by representing the quasiperiodic function
  $\log \| B^{(n)} (.)\|$ as a section through a doubly $1$-periodic 
  function in two variable (as
  in Fact~\ref{fact:quasi}). Due to taking the infimum over $n\in\NN$, 
  it does not matter which matrix norm is used in the representation 
  as an integral over the $2$-torus.
\end{proof}

\begin{remark}
  Numerical experiments and various other approaches suggest that more
  is true than what we stated in Proposition~\ref{prop:L-out}.  In
  particular, we think that the extremal Lyapunov exponents actually
  exist as limits for a.e.\ $k\in\RR$.  At each such\/ $k$, one would
  then have a matching vector space filtration
\[
     \{ 0 \}  =  E^{(0)}_{+} (k) \, \subsetneq \,
      E^{(1)}_{+} (k) \, \subsetneq \,
      E^{(2)}_{+} (k)  = \CC^{2}
\]  
with the equivariance condition\/
$E^{(1)}_{+} \bigl( \frac{k}{\lambda}\bigr) = B\bigl(
\frac{k}{\lambda}\bigr) \, E^{(1)}_{+} (k)$,
the latter as a consequence of the recursion in conjunction with
Lemma~\ref{lem:constant}. It is not clear, however, what the precise
relation with the existing filtration of the inward iteration would
be.  \exend
\end{remark}

As mentioned earlier, the exponents for Eq.~\eqref{eq:out-iter}, with
the extra factor of $\sqrt{\lambda}$ in front of each matrix, are
given by $\chi^{(i)}_{+} (k) = \tilde{\chi}^{(i)}_{+} (k) + \log
\sqrt{\lambda}$. Both versions exist as limits under the same 
conditions. For the bounds that matter to us, one has the
following consequence, where we only use the definition of
the exponents according to Eq.~\eqref{eq:limsup}.

\begin{coro}\label{coro:positive}
  Under the assumptions of Proposition~\emph{\ref{prop:L-out}}, and
  for a.e.\ $k\in\RR$, the minimal Lyapunov exponent\/
  $\chi^{(1)}_{+} (k)$ for the outward iteration of
  Eq.~\eqref{eq:out-iter} is bounded below by the constant\/
  $D = \frac{1}{2} \log (\lambda) - \tilde{D}$, where\/
  $\tilde{D}$ is the constant from
  Proposition~\emph{\ref{prop:L-out}}.  Moreover, one has\/ $D >0$,
  which means that both Lyapunov exponents are strictly positive
  almost everywhere.
\end{coro}

\begin{proof}
  The first two claims are immediate consequences of
  Proposition~\ref{prop:L-out}. 

  For the positivity of $D$, observe that, for any $n\in\NN$, we have
\[
   \bigl|\tilde{\chi}^{(1)}_{+} (k)\bigr| \, \leqslant \, 
    \myfrac{1}{n}\, \MM \bigl(  \log  \big\| 
   B^{(n)} (.)  \big\| \bigr) ,
\]
where the estimate holds for a.e.\ $k\in\RR$.

Now, we calculate the mean on the right-hand side 
for the first few integers, by precise numerical
integration, where we may use the Frobenius norm
and the last identity from Proposition~\ref{prop:L-out}.
For $n=4$, this gives an upper bound for $\lvert 
\tilde{\chi}^{(1)}_{+} (k) \rvert$ of $0.385 (1)$, which is 
smaller than $\log \sqrt{\lambda} \approx 0.417$. Since
\[
  \chi^{(1)}_{+} (k) \, = \, \log \mbox{\small $\sqrt{\lambda}$}
   \, - \, \lvert \tilde{\chi}^{(1)}_{+} (k) \rvert
  \quad \text{and} \quad
  \chi^{(1)}_{+} (k) + \chi^{(2)}_{+} (k)
  \, =  \, \log (\lambda)
\]
for a.e.\ $k\in\RR$, one has $0 < \chi^{(1)}_{+} (k) \leqslant
\chi^{(2)}_{+} (k)$, and our claim follows.
\end{proof}

\subsection{Conclusions}

The key point now is that our finding of positive Lyapunov exponents
for the outward iteration is not compatible with translation
boundedness of the diffraction measures. Indeed, for any
$u\in\CC^{2}$, $\widehat{\gamma^{}_{u}}$ is a translation bounded,
positive measure.  This also implies that its absolutely continuous
part is translation bounded. Then,
$ \bigl(\widehat{\gamma^{}_{(1,0)}} + \widehat{\gamma^{}_{(0,1)}}
\bigr)_{\mathsf{ac}}$
is translation bounded as well, and this measure is represented by the
Radon--Nikodym density
$ g(k) = \lvert v^{}_{1} (k)\rvert^{2} + \lvert v^{}_{2} (k)\rvert^{2}
$.
If $D$ is the lower bound of $\chi^{(1)}_{+} (k)$ for a.e.\ $k$ from
Corollary~\ref{coro:positive}, we
know that, for any $\delta > 0$, there is a constant $C>0$ such that
\[
    g (\lambda^{n} k) \, \geqslant \, C \, \ee^{2 (D - \delta) n} \ts g (k)
\]
holds for a.e.\ $k$; compare \cite[p.~20]{BP1}.  Since $D>0$, we may
choose $\delta = D/2$. Now, the constant $C$ might still depend on
$k$, but we can employ the argument from Lemma~\ref{lem:eps-suffice}
here. Indeed, if we fix some $\varepsilon >0$, we may restrict $k$ to
the interval $[0, \varepsilon]$. It suffices to show that $g=0$ in the
Lebesgue sense on this interval. Now, we may assume without
loss of generality that $C(k)$ defines a measurable function on
$[0,\varepsilon]$, where it satisfies $C(k) >0$ on a subset of full
measure. We can now invoke Lemma~\ref{lem:out-growth-simple}
from Appendix B to conclude $g=0$ in the Lebesgue sense on 
$[0,\varepsilon]$, and then also on $\RR_{+}$ by 
Lemma~\ref{lem:eps-suffice}. This gives the following result.

\begin{theorem}\label{thm:no-ac}
  The absolutely continuous parts of\/ $\widehat{\vU}_{ij}$
  must vanish, which is to say that\/ $h (k) =0$ for
  a.e.\ $k\in\RR$. Consequently, $\bigl( \widehat{\gamma^{}_{u}}
  \bigr)_{\mathsf{ac}} = 0$ for all\/ $u \in \CC^{2}$.   \qed
\end{theorem}

It is clear that $\widehat{\vU}$, as a vector measure, cannot be pure
point. This follows from the simple observation that the inverse
transform of $\bigl(\widehat{\vU} \ts\ts \bigr)_{\mathsf{pp}} = \cI
(0) \, \delta^{}_{0}$ does not have discrete support. So, the
following conclusion is obvious.

\begin{coro}\label{coro:singular}
  If\/ $\vU_{ij}$ is one of our pair correlation measures, its Fourier
  transform is of the form
\[
   \widehat{\vU}_{ij} \, = \, 
    \frac{\dens (\vL^{(\alpha)}) \ts \dens (\vL^{(\beta)})}
       {\bigl(\dens (\vL)\bigr)^{2}} \, \delta^{}_{0} +
    \bigl(\widehat{\vU}_{ij}\bigr)_{\mathsf{sc}} \ts ,
\]
   with non-trivial singular continuous part.   \qed
\end{coro}

Let us now turn to the actual diffraction and see how we can
use this insight to calculate the diffraction measure, at least
numerically and with good precision.

\section{Application to diffraction and outlook}\label{sec:appl}

\subsection{Example with balanced weights}
Consider 
\begin{equation}\label{eq:bal-measure}
    \omega^{}_{0} \; = \; 
    \sum_{x\in\vL} u(x)\, \delta_{x}
\end{equation}
with weights $u(x)\in\{u^{}_{0},u^{}_{1}\}$ such that the frequency
average its $u^{}_{0}\ts \nu^{}_{00}(0)+u^{}_{1}\nu^{}_{11}(0)=0$, and
hence $I^{}_{0}=0$ in Theorem~\ref{thm:pp}. One possible choice, which
we will adopt here, is given by the real weights $u^{}_{0}=1-\lambda$
and $u^{}_{1}=1$. We can now express the autocorrelation
$\gamma^{}_{0}$ of $\omega^{}_{0}$ by Eq.~\eqref{eq:new-gamma}, which
is of the form $\gamma^{}_{0} = \sum_{z\in\varDelta} \eta^{}_{0} (z) \,
\delta^{}_{z}$ with $\varDelta$ as in Proposition~\ref{prop:hull} and
\[
    \eta^{}_{0}(z) \, = \, \dens (\vL^{w})\,
    \bigl(1-\lambda,1\bigr) 
    \begin{pmatrix} \nu^{}_{00}(z) &  \nu^{}_{01}(z) \\
     \nu^{}_{10}(z) &  \nu^{}_{11}(z) \end{pmatrix}
     \begin{pmatrix} 1-\lambda \\ 1 \end{pmatrix} .
\]
In particular, we have
\begin{equation}\label{eq:slope}
    \gamma^{}_{0} \bigl( \{ 0 \} \bigr)
    \, = \: \eta^{}_{0} (0) \, = \,
    \dens (\vL^{w}) \, (\lambda - 1) \, = \,
    \myfrac{6 \lambda - 3}{13} \, \approx \, 0.832 \ts .
\end{equation}

The Fourier transform, $\widehat{\ts\gamma^{}_{0}\ts}$, is of the form
given in Eq.~\eqref{eq:u-FT}. It is a positive and translation bounded
measure. Moreover, due to the balanced weights, it has no pure point
part, so $\widehat{\ts\gamma^{}_{0}\ts}(\{k\})=0$ for all $k\in\RR$.
In fact, as a consequence of Corollary~\ref{coro:singular},
$\widehat{\ts\gamma^{}_{0}\ts}$ is then purely singular
continuous. Figure~\ref{fig:dist} illustrates the corresponding
distribution function $F$, where
\begin{equation}\label{eq:dist}
    F(x) \, := \, \widehat{\ts\gamma^{}_{0}\ts} 
    \bigl( [0,x]\bigr),
\end{equation}
for $x\in [0,3]$. This was calculated numerically via the integration
of approximating trigonometric sums, the latter being the absolute
squares of exponential sums that were obtained by an exact iteration
based on the inflation.  Alternatively, they can also be calculated on
the basis of a matrix Riesz product, in the spirit of
\cite[Sec.~2]{BS}.

\begin{figure}
\begin{center}
  \includegraphics[width=0.75\textwidth]{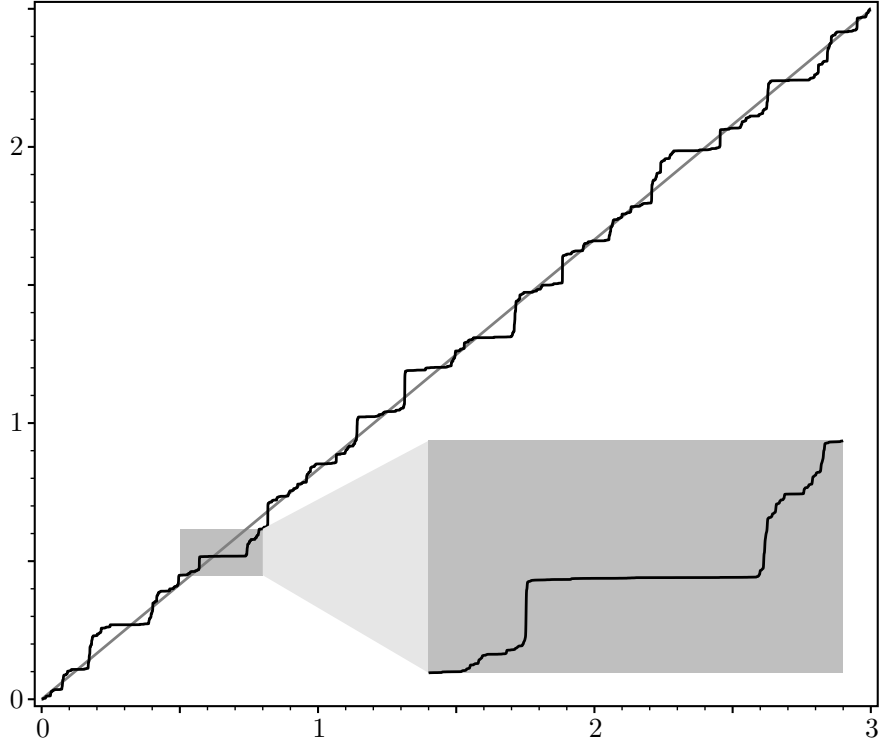}
\end{center}
\caption{\label{fig:dist} Sketch of the distribution function $F$ from
  Eq.~\eqref{eq:dist} for the singular continuous diffraction measure
  in the balanced weight case of Eq.~\eqref{eq:bal-measure}. The
  straight line is the average slope of $F$, which is $\eta^{}_{0}(0)$
  of Eq.~\eqref{eq:slope}. The enlarged detail is meant to illustrate
  that $F$ is strictly increasing, which implies that there are no
  proper plateaux.}
\end{figure}

\begin{remark}
  On average, the distribution function $F(x)$ grows linearly in $x$,
  with slope $\eta^{}_{0} (0)$, as also indicated in
  Figure~\ref{fig:dist}. Furthermore, $F$ is a continuous function
  that appears to be strictly increasing, in line with the general
  expectation that the supporting set of such measures must be 
  dense; compare \cite[Sec.~10.1]{TAO}.
\exend
\end{remark}

\subsection{Outlook}
Our main result concerns the absence of absolutely continuous
components in the diffraction measure, for arbitrary complex
weights. Now, one would like to also infer the absence of absolutely
continuous spectral measures for the dynamical system. This would
follow if one can also show that patch derived factors have no
absolutely continuous diffraction components, as an application of
\cite[Thm.~15 and Cor.~16]{BLvE}.

More generally, our analysis can be applied to other binary systems in
complete analogy, where the result will be that an irreducible IDA in
conjunction with a certain non-uniform hyperbolic structure in the
matrix iterations is incompatible with the presence of absolutely
continuous diffraction components.  Cases with reducible IDAs
can also be handled by explicit methods \cite{Neil}.
The situation for primitive
inflation rules over larger alphabets, or for inflation tilings in
higher dimensions, is more complex, and will be considered
separately.

\section{Appendix A: Details of the inward iteration}

Here, we study the iteration according to Eq.~\eqref{eq:newiter} in more
detail, where it suffices to consider $k \geqslant 0$.  First, choose
$c^{}_{\ts 0}>0$ such that $\lvert z(k) \rvert \leqslant 9 \ts c^{}_{\ts 0} 
k$ for all $k\in \bigl[ 0, \frac{1}{3} \bigr]$, which is possible because 
$z$ is analytic, with $z(k) = \cO (k)$ for $k \to 0$.  Now, set
$v^{\ts \prime}_{n} = ( \alpha^{}_{n}, \beta^{}_{n})^{t}$ and consider,
for some fixed $0 < k \leqslant \frac{1}{3}$, the recursion
\begin{equation}\label{eq:new-rec}
     v^{\ts \prime}_{n} \, = \, \myfrac{1}{\mbox{\small $\sqrt{\lambda}$}}\,
   \bigl( D - z^{}_{n} \ts N \bigr) \ts v^{\ts \prime}_{n-1}
\end{equation}
with $z^{}_{n} = z \bigl( \frac{k}{\lambda^{n}} \bigr)$ for $n\in\NN$,
with arbitrary start vector $v^{\ts \prime}_{0} \ne 0$. This means that
$(\alpha^{}_{n}, \beta^{}_{n})^{t}$ stands for $\bigl(\alpha
(k/\lambda^{n}), \beta (k/\lambda^{n}) \bigr)^{t}$ in the previous
iteration of Eq.~\eqref{eq:newiter}. Observe next that $\det(D -
z^{}_{n} \ts N) = \det\bigl(B(k/\lambda^{n}) \bigr) \ne 0$ for all
$n\in\NN$ under our condition on $k$, wherefore we know that
$v^{\ts \prime}_{n} \ne 0$ for all $n\geqslant 0$. Also, we have $\lvert
z^{}_{n} \rvert \leqslant 9 \ts c^{}_{\ts 0} k \ts \lambda^{-n}
\leqslant 3 \ts c^{}_{\ts 0} \lambda^{-n}$ for all $n\in\NN$. Now,
with $\| D \|^{}_{\infty} = \lambda$ and $\| N \|^{}_{\infty} =
\lambda/3$, we can estimate
\[
   \| v^{\ts \prime}_{n} \|^{}_{\infty} \, \leqslant \, \sqrt{\lambda}
   \, \bigl( 1 + c^{}_{\ts 0}  \lambda^{-n} \bigr) 
   \| v^{\ts \prime}_{n-1} \|^{}_{\infty}
   \leqslant \dots \leqslant \lambda^{n/2}
   \bigl( 1 +  c^{}_{\ts 0}  \lambda^{-n} \bigr) 
   \cdot \ldots \cdot \bigl( 1 +  c^{}_{\ts 0} 
   \lambda^{-1} \bigr) \| v^{\ts \prime}_{0} \|^{}_{\infty} \ts ,
\]
which implies $\| v^{\ts \prime}_{n} \|^{}_{\infty} = \cO(\lambda^{n/2})$
because the product $\prod_{m\geqslant 1} \bigl(1 + c^{}_{\ts 0}
\lambda^{-m} \bigr)$ is absolutely convergent. Since $\lvert
\alpha^{}_{n} \rvert \leqslant \| v^{\ts \prime}_{n} \|^{}_{\infty}$, this
also means
\[
    \alpha^{}_{n} \, = \, \cO (\lambda^{n/2}) \quad 
    \text{as $n\to \infty$} \ts ,
\]
as well as $\| v^{\ts \prime}_{n} \|^{}_{\infty} \leqslant c^{\ts\ts \prime}
\lambda^{n/2}$ for some $c^{\ts\ts \prime}> 0$ and all $n\in\NN_{0}$.

Now, choose $\varepsilon>0$ small enough so that $\lvert z^{}_{n}
\rvert \leqslant 3 \ts c\ts \lambda^{-n}$ holds for all $k\in
[0,\varepsilon]$ with
\begin{equation}\label{eq:choose-C}
   c \, = \, c(\varepsilon) \, < \,
   \min \big\{\lambda - \sqrt{\lambda},
   1+\sqrt{\lambda}-\lambda \big\} \, = \,
   1+\sqrt{\lambda}-\lambda \, \approx \,
   0.2147 \ts ,
\end{equation}
which is clearly possible because $c=3\ts c^{}_{\ts 0}\varepsilon$. For
any $0 < k \leqslant \varepsilon$ and $n\in\NN_{0}$, we then have
\begin{equation}\label{eq:iter-bound}
    \Big\| \frac{z^{}_{n+1}}{\mbox{\small $\sqrt{\lambda}$}}
    \, N \ts v^{\ts \prime}_{n} \ts \Big\|_{\infty} \ts \leqslant \,
    \myfrac{c\ts \lambda^{-n}}{\mbox{\small $\sqrt{\lambda}$}} \,
    \bigl\| v^{\ts \prime}_{n} \bigr\|_{\infty} \, \leqslant
    \,  \myfrac{c}{\mbox{\small $\sqrt{\lambda}$}} \,
    \bigl\| v^{\ts \prime}_{n} \bigr\|_{\infty} \, < \,
    \myfrac{\lambda - \mbox{\small $\sqrt{\lambda}$}}
     {\mbox{\small $\sqrt{\lambda}$}} \,
    \bigl\| v^{\ts \prime}_{n} \bigr\|_{\infty} \ts .
\end{equation}
The first estimate, together with Eq.~\eqref{eq:new-rec}, implies
\[
   \lvert \beta^{}_{n+1} \rvert \, \leqslant \,
   \myfrac{\lambda - 1}{\mbox{\small $\sqrt{\lambda}$}} \,
   \lvert \beta^{}_{n} \rvert
   + \myfrac{c \ts\lambda^{-n}}{\mbox{\small $\sqrt{\lambda}$}} \,
   \| v^{\ts \prime}_{n} \|^{}_{\infty} \, \leqslant \,
   \myfrac{\lambda - 1}{\mbox{\small $\sqrt{\lambda}$}} \, 
   \lvert \beta^{}_{n} \rvert
   + \myfrac{c^{\ts \ts \prime\prime} \lambda^{-n/2}}
          {\mbox{\small $\sqrt{\lambda}$}}
\]
with $c^{\ts\ts \prime\prime} = c\ts c^{\ts\ts \prime}$. Iterating this estimate
inductively gives
\[
\begin{split}
   \lvert \beta^{}_{n+1} \rvert \, & \leqslant \,
   \frac{(\lambda - 1)^{n+1}}
     {\lambda^{(n+1)/2}} \, \lvert \beta^{}_{0} \rvert +
   \myfrac{c^{\ts \ts \prime\prime}}{\lambda^{(n+1)/2}} 
   \sum_{m=0}^{n} (\lambda - 1)^{m} \\[1mm]
    & \leqslant \, \lvert \beta^{}_{0} \rvert \,
   \frac{(\lambda - 1)^{n+1}}{\lambda^{(n+1)/2}}
   + \myfrac{c^{\ts \ts \prime\prime}}{\lambda - 2} \,
   \frac{(\lambda-1)^{n+1} - 1}{\lambda^{(n+1)/2}} \\[2mm]
   & < \, \left(\lvert \beta^{}_{0} \rvert
   + \myfrac{c^{\ts \ts \prime\prime}}{\lambda - 2} \right)
   \left( \frac{\lambda - 1}{\mbox{\small $\sqrt{\lambda}$}} 
   \right)^{\! n+1} \; = \; \cO \left( \Bigl( 
   \myfrac{3}{\lambda^{3/2}}  \Bigr)^{\! n+1} \ts \right),
\end{split}
\] 
where $(\lambda - 1)/\mbox{\small $\sqrt{\lambda}$} = 3/\lambda^{3/2}
\approx 0.8585 < 1$.
Consequently, we have $\lim_{n\to\infty} \beta^{}_{n} = 0$, for any
start vector $v^{}_{0}$ and any $0<k\leqslant \varepsilon$.

If $\lvert \alpha^{}_{n} \rvert > \lvert \beta^{}_{n} \rvert$
for some $n$, which is a situation that can already occur for
$v^{\ts \prime}_{0}$, we have $\| v^{\ts \prime}_{n} \|^{}_{\infty} = \lvert
\alpha^{}_{n} \rvert > 0$, and the iteration \eqref{eq:new-rec} with the
estimates from Eqs.~\eqref{eq:choose-C} and \eqref{eq:iter-bound}
result in
\[
   \lvert \alpha^{}_{n+1} \rvert \, \geqslant 
   \lvert \alpha^{}_{n} \rvert \, 
   \myfrac{\lambda - c}{\mbox{\small $\sqrt{\lambda}$}} \, > \,
   \lvert \alpha^{}_{n} \rvert
   \quad \text{and} \quad
   \lvert \beta^{}_{n+1} \rvert \, \leqslant 
   \myfrac{\lambda - 1}{\mbox{\small $\sqrt{\lambda}$}} \,
   \lvert \beta^{}_{n} \rvert  +  
    \myfrac{c}{\mbox{\small $\sqrt{\lambda}$}} \,
    \lvert \alpha^{}_{n} \rvert \, < \, 
   \lvert \alpha^{}_{n} \rvert \ts ,
\]
so that
$\lvert \alpha^{}_{n+1} \rvert > \lvert \beta^{}_{n+1} \rvert$. By
induction,
$\lvert \alpha^{}_{m+1} \rvert > \lvert \alpha^{}_{m} \rvert > 0$ for
all $m\geqslant n$, while $\beta^{}_{m}$ goes to $0$ as before.  Using
the first inequality in Eq.~\eqref{eq:iter-bound}, we get
$\lvert \alpha^{}_{n+1} \rvert \geqslant \mbox{\small $\sqrt{\lambda}$} \, \lvert
\alpha^{}_{n} \rvert \ts \bigl( 1 - c\ts \lambda^{-n-1} \bigr)> 0$
and hence, for any $m\in\NN$,
\[
    \lvert \alpha^{}_{n+m} \rvert \, \geqslant \, \lambda^{m/2} \,
    \lvert \alpha^{}_{n} \rvert \, \prod_{\ell=1}^{m}
    \bigl(1 - c \ts\lambda^{-n-\ell} \bigr).
\]
Since the product converges absolutely as $m\to\infty$, we may
conclude that $\lvert \alpha^{}_{m} \rvert \geqslant c^{}_{\alpha}
\ts\lambda^{m/2}$ for all $m\geqslant n$, with some constant
$c^{}_{\alpha}>0$ that depends on $k$ and $v^{\ts \prime}_{0}$. With our
previous estimate, we thus have $\| v^{\ts \prime}_{n}\|^{}_{\infty} =
\lvert \alpha^{}_{n} \rvert \asymp \lambda^{n/2}$, which means that
$\lvert \alpha^{}_{n} \rvert$ is bounded from above and from below by
different (positive) multiples of the same exponential function. For
any initial vector $v^{\ts \prime}_{0}$ with this behaviour, we thus
obtain the Lyapunov exponent
\begin{equation}\label{eq:in-upper}
    \chi^{}_{-} ( v^{\ts \prime}_{0} ) \, := \,
    \limsup_{n\to\infty} \myfrac{1}{n} \ts \log 
    \| v^{\ts \prime}_{n} \|^{}_{\infty} \, = \,
    \lim_{n\to\infty} \myfrac{1}{n} \ts \log 
    \| v^{\ts \prime}_{n} \|^{}_{\infty} \, = \,
    \log \mbox{\small $\sqrt{\lambda}$} \, \approx \, 0.834 \ts ,
\end{equation}
where the existence of the limit is a simple consequence of the
asymptotic behaviour.

It remains to consider the case that $\lvert \alpha^{}_{n} \rvert
\leqslant \lvert \beta^{}_{n} \rvert$ for all $n\in\NN_{0}$, which
implies $\| v^{\ts \prime}_{n} \|^{}_{\infty} = \lvert \beta^{}_{n}
\lvert$. If $\alpha^{}_{n} = 0$ for some $n$, we get $\alpha^{}_{n+1}
\ne 0$ from the structure of $N$, and hence also $\beta^{}_{n+1} \ne
0$. Note that this is just the irreducibility of $\cB$ in action;
compare Lemma~\ref{lem:ida}. Choose an $n\in\NN_{0}$ such that
$\beta^{}_{n} \ne 0$ and observe that we then get
\[
  0 \, < \, \lvert \beta^{}_{n} \rvert \,
   \myfrac{\lambda - 1}{\mbox{\small $\sqrt{\lambda}$}} \, \Bigl(1 -  
   \myfrac{c}{\lambda - 1} \, \lambda^{-n} \Bigr)
   \, \leqslant \,  \lvert \beta^{}_{n+1} \rvert \, \leqslant \,
   \lvert \beta^{}_{n} \rvert \,
   \myfrac{\lambda - 1}{\mbox{\small $\sqrt{\lambda}$}} \, \Bigl(1 +
   \myfrac{c}{\lambda - 1} \, \lambda^{-n} \Bigr)
\]
because $0<c<\lambda - 1$
with the constant $c$ from Eq.~\eqref{eq:choose-C}. 
The upper estimate implies that $\lvert \beta^{}_{n} \rvert$
is ultimately monotonically decreasing.
Iterating the lower estimate leads to
\[
   \lvert \beta^{}_{n+m} \rvert \, \geqslant \,
  \left( \myfrac{\lambda - 1}{\mbox{\small $\sqrt{\lambda}$}} 
   \right)^{\! m}  \lvert \beta^{}_{n} \rvert \,
   \prod_{\ell = 0}^{m-1} \Bigl( 1 - \myfrac{c}{\lambda - 1} \,
   \lambda^{-n-\ell}\Bigr) 
\]
for all $m\geqslant 1$. Since the product converges absolutely as
$m\to\infty$, we may conclude that $\| v^{\ts \prime}_{n} \|^{}_{\infty} =
\lvert \beta^{}_{n} \rvert \asymp \bigl( 3 / \lambda^{3/2} \bigr)^{n}$
as $n\to\infty$, hence also $\lim_{n\to\infty} v^{\ts \prime}_{n} = 0$.

Let us look at a consequence of this asymptotic behaviour on the
convergence rate of $\alpha^{}_{n}$. Select an $n$ with $\alpha^{}_{n}
\ne 0$, which we know to exist. With $\| v^{\ts \prime}_{n}
\|^{}_{\infty}=\lvert \beta^{}_{n}\rvert$, we then find
\[
   \lvert \alpha^{}_{n+1} \rvert \, \geqslant \, 
   \mbox{\small $\sqrt{\lambda}$} \, \lvert \alpha^{}_{n} \rvert
   - \myfrac{c}{\lambda^{n+1/2}} \, \lvert \beta^{}_{n} 
   \rvert \, \geqslant \, \lvert \alpha^{}_{n} \rvert +
   \bigl(\mbox{\small $\sqrt{\lambda}$} - 1\bigr) \ts
   \lvert \alpha^{}_{n} \rvert  \, - \, \widetilde{c}\,
   \Bigl(\myfrac{3}{\lambda^{5/2}}\Bigr)^{\! n}
\]
for some constant $\widetilde{c}$ that depends on $k$ and
$\beta^{}_{0}$.  When $\bigl(\mbox{\small $\sqrt{\lambda}$} - 1\bigr)
\ts \lvert \alpha^{}_{n} \rvert \geqslant \widetilde{c} \, \bigl(
3/\lambda^{5/2}\bigr)^{\nts n}$, one gets
$\lvert \alpha^{}_{n+1} \rvert \geqslant \lvert \alpha^{}_{n} \rvert >
0$ and then, inductively,
$\lvert \alpha^{}_{n+m}\rvert \geqslant \lvert \alpha^{}_{n} \rvert >
0$ for all $m\in\NN$, which contradicts
$\lim_{m\to\infty} \alpha^{}_{m} = 0$.  Consequently, whenever
$\alpha^{}_{n} \ne 0$ in our present case, we must have
\[
     \lvert \alpha^{}_{n} \rvert \, < \, 
     \myfrac{\widetilde{c}}{\mbox{\small $\sqrt{\lambda}$} - 1} \,
     \Bigl( \myfrac{3}{\lambda^{5/2}} \Bigr)^{\! n}
\]   
which altogether means $\alpha^{}_{n} = \cO \bigl( 3/ \lambda^{5 n/2}
\bigr)$, as well as $\lvert \alpha^{}_{n} / \beta^{}_{n} \rvert = \cO
(\lambda^{-n})$.

The fact that also this somewhat counterintuitive behaviour does
indeed occur can be seen as follows. Under our restriction on $k$, we
know that the mapping $v^{\ts \prime}_{0} \mapsto v^{\ts \prime}_{n} =
Q^{}_{n} (k) \ts v^{\ts \prime}_{0}$ is invertible. Consequently, it is
possible to choose $v^{\ts \prime}_{0} \in \partial B^{}_{1} (0)$ in such
a way that $v^{\ts \prime}_{n} = c^{}_{n} \ts (0,1)^{t}$ for some
$c^{}_{n} > 0$, which is the contracting direction for $M$ in this
representation. Let $v^{\ts \prime}_{0,n}$ be this initial condition, and
consider the sequence $\bigl( v^{\ts \prime}_{0,n} \bigr)_{n\in\NN}$
constructed this way. As it lies within the compact set $\partial
B^{}_{1} (0)$, there is a converging subsequence, with limit
$v^{\ts \prime}_{0,\infty}$, say. By construction, this vector spans a
one-dimensional subspace $E^{\ts \prime}_{-} (k)$ of vectors that behave
as just derived under the inward iteration. For any $v^{\ts \prime}_{0}
\in E^{\ts \prime}_{-} (k)$, we thus get
\begin{equation}\label{eq:in-lower}
    \chi^{}_{-} ( v^{\ts \prime}_{0} ) \, = \,
    \lim_{n\to\infty} \myfrac{1}{n} \ts \log 
    \| v^{\ts \prime}_{n} \|^{}_{\infty} \, = \,
    \log \myfrac{\lambda - 1}{\mbox{\small $\sqrt{\lambda}$} }
    \, = \, \log \myfrac{3}{\lambda^{3/2}}     
    \, \approx \, -0.153 \ts ,
\end{equation}
where the existence of the Lyapunov exponent as a limit follows
as in Eq.~\eqref{eq:in-upper}.

\section{Appendix B: Two simple scaling arguments}

Let us briefly explain the idea to use an `impossible' singularity of
a locally integrable function to conclude that such a function must
vanish, spelled out in a one-dimensional setting for
illustration. Here, for $\lambda > 1$ fixed, $h \bigl(
\frac{x}{\lambda}\bigr) = \lambda \ts\ts h (x)$ would be the (scalar)
analogue of Eq.~\eqref{eq:h-iter}. This implies $h(x) = \mathcal{O}
(x^{-1})$ as $x\to 0$, a contradiction to local integrability unless
$h=0$. The precise statement is the following.

\begin{lemma}
  Let\/ $h\in L^{1} \bigl([0,1], \CC \bigr)$, and assume that,
  for some fixed\/ $\lambda > 1$, one has
\[
   h \bigl( \tfrac{x}{\lambda}\bigr) \, = \, \lambda \, h(x)
\]
  for a.e.\ $x\in [0,1]$. Then, $h=0$, which means\/
  $h(x) = 0$ a.e.\ in\/ $[0,1]$.
\end{lemma}

\begin{proof}
  By assumption, $h$ is certainly integrable on
  $\bigl[\frac{1}{\lambda},1 \bigr]$. Since
\[
   h \bigl( \tfrac{x}{\lambda^{m}} \bigr) \, = \,
   \lambda^{m} \ts h(x)
\]
also holds for all $m\in\NN$ and still for a.e.\ $x\in [0,1]$,
one has
\[
   \int_{1/\lambda^{m+1}}^{1/\lambda^{m}} \lvert h(x) \rvert \dd x
  \, = \, \myfrac{1}{\lambda^{m}} \int_{1/\lambda}^{1}
  \big\lvert h \bigl( \tfrac{y}{\lambda^{m}} \bigr) 
   \big\rvert \dd y \, = \int_{1/\lambda}^{1} \lvert h(y) 
   \rvert \dd y \ts ,
\]
via a simple transformation of variable ($y = \lambda^{m}\ts x$),
hence
\[
    \| h \|^{}_{1} \, =  \sum_{m\ge 0} 
    \int_{1/\lambda^{m+1}}^{1/\lambda^{m}}
     \lvert h(x) \rvert \dd x \, =
    \sum_{m\ge 0} \int_{1/\lambda}^{1} 
   \lvert h(x) \rvert \dd x \ts .
\]
Consequently, $\| h \|^{}_{1} < \infty$ implies $\lvert h(x) \rvert =
0$ a.e. on $\bigl[\frac{1}{\lambda},1 \bigr]$. The functional equation
then gives $\lvert h(x) \rvert = 0$ a.e. on $\bigl[\frac{1}
{\lambda^{m+1}}, \frac{1}{\lambda^{m}} \bigr]$ for any $m\ge 0$, hence
$h=0$ as claimed.
\end{proof}

Let us also discuss why a positive, pointwise defined 
Lyapunov exponent for the
asymptotic growth of a density function $g \geqslant 0$ is
incompatible with translation boundedness of the measure
$g \mu^{}_{\mathrm{L}}$ defined by $g$, where
$\mu^{}_{\mathrm{L}}$ denotes Lebesgue measure on $\RR$.

\begin{lemma}\label{lem:out-growth}
  Let\/ $g \in L^{1}_{\mathrm{loc}} (\RR_{+})$ be a non-negative
  function and let\/ $\lambda > 1$ be fixed. Assume further that\/ $g$
  satisfies the relation\/
  $g (\lambda x) \geqslant \vartheta (x) \ts g (x)$ for a.e.\ $x>0$,
  where\/ $\vartheta$ is a bounded, measurable function with\/
  $\vartheta (x) >1$ and\/ $\vartheta (\lambda x) = \vartheta (x)$ for
  a.e.\ $x>0$. Then, the absolutely continuous measure\/
  $g\ts \mu^{}_{\mathrm{L}}$ is translation bounded if and only if\/
  $g=0$ in the Lebesgue sense.
\end{lemma}

\begin{proof}
  If $g$ does not vanish almost everywhere, we have $q:=
  \int_{0}^{a} g (x) \dd x >0$ for some $a>0$, and $f := q^{-1} g$
  is a probability density on $[0,a]$. By assumption, we have
\[
     g (\lambda^{\nts m} x) \, \geqslant \, 
     \vartheta (\lambda^{\nts m-1} x)
     \cdot \ldots \cdot \vartheta (x) \, g(x) \, = \,
     \vartheta (x)^{m} \ts g(x)
\]  
for any (fixed) $m\in\NN$ and a.e.\ $x>0$, by an iterated application
of the assumed estimate. Now, via the (clearly existing) moments of
$\vartheta$ relative to $f$, we get
\[
\begin{split}
    \int_{0}^{a \lambda^{\nts m}} \!\! g(x) \dd x \,
    & = \, \lambda^{\nts m}
    \int_{0}^{a} g (\lambda^{\nts m} y) \dd y \, \geqslant \,
    q \ts \lambda^{\nts m} \int_{0}^{a}  \vartheta (y)^{m}
    \, f(y) \dd y \\[2mm]
    & = \, q \ts \lambda^{\nts m} \, \EE_{f} \bigl( \vartheta^{m} \bigr)
    \, \geqslant \, q \ts \lambda^{\nts m} \bigl( \EE_{f} (\vartheta) \bigr)^{m} 
    \, = \, q \, (\varTheta_{1} \lambda)^{m} ,
\end{split}
\]
where the estimate in the second line is a consequence of Jensen's
inequality, while the first moment,
$\varTheta_{1} = \EE_{f} (\vartheta)$, clearly satisfies
$\varTheta_{1} >1$.

This shows that $ \int_{0}^{a\lambda^{\nts m}} g(x) \dd x \not\in 
\cO (\lambda^{\nts m} )$ as $m\to\infty$. However, 
$g\ts \mu^{}_{\mathrm{L}}$ translation bounded with $g\geqslant 0$
implies
\[
     \int_{0}^{L} g(x) \dd x \, = \, \cO (L) \quad \text{as } L\to\infty \ts ,
\]
which contradicts the previous estimate. Therefore, we must have
$g(x) = 0$ for a.e.\ $x>0$ as claimed.
\end{proof}

A simple variant, which applies to cases with a.e.\ constant Lyapunov
exponents as in our example, can be stated as follows.

\begin{lemma}\label{lem:out-growth-simple}
  Let\/ $g \in L^{1}_{\mathrm{loc}} (\RR_{+})$ be a non-negative
  function and let\/ $\lambda > 1$ be fixed. Assume further that
  there is an interval\/ $[0,a]$ with\/ $a>0$, a constant\/ $\vartheta>1$,
  and a measurable function\/ $C$ with\/ $C(x)>0$ for a.e.\
  $x\in [0,a]$ such that\/
  $g(\lambda^m x) \geqslant C(x) \, \vartheta^{m} \ts g(x)$
  holds for a.e.\ $x\in [0,a]$. 
  Then, the absolutely continuous measure\/
  $g\ts \mu^{}_{\mathrm{L}}$ is translation bounded if and only if\/
  $g=0$ on\/ $[0,a]$ in the Lebesgue sense.
\end{lemma}

\begin{proof}
Assume to the contrary of our claim that $g(x) > 0$ for a subset
of $[0,a]$ of positive measure. Then, 
\[
     c^{}_{g} \, := \int_{0}^{a}  C(x) \, g(x) \dd x \, > \, 0 
\]
because the set $\{ x \in [0,a] : C(x) = 0 \text{ or } g(x) = 0 \}$
is not of full measure in $[0,a]$. Consequently, one has
\[
   \int_{0}^{a \lambda^m} \! g(x) \dd x \, \geqslant \,
   \lambda^m  \ts \vartheta^m \int_{0}^{a} C(x) \, g(x) \dd x
   \, = \, c^{}_{g} \, \lambda^m \ts  \vartheta^m
\]
and thus obtains the same type of contradiction as in the
proof of Lemma~\ref{lem:out-growth}. So, we must have
$g=0$ on $[0,a]$ in the Lebesgue sense as claimed.
\end{proof}

\section*{Acknowledgements}
It is our pleasure to thank Yann Bugeaud, David Damanik,
Franz G\"{a}hler, Alan Haynes, Andrew Hubery,
Neil Ma\~{n}ibo and Nicolae Strungaru for helpful
discussions. This work was supported by the German Research Council
(DFG), within the CRC 701.


\begin{thebibliography}{99}

\bibitem{Arga}
Argabright L and Gil de Lamadrid J,
Fourier analysis of unbounded measures on locally compact Abelian groups,
\textit{Memoirs Amer.\ Math.\ Soc.}, no.\ 145 (1974).

\bibitem{BG15}
Baake M and G\"{a}hler F,
Pair correlations of aperiodic inflation rules via renormalisation: 
Some interesting examples,
\textit{Topol.\ Appl.} \textbf{205} (2016) 4--27; 
\texttt{arXiv:1511.00885}.

\bibitem{TAO}
Baake M and Grimm U,
\textit{Aperiodic Order. Vol.\ 1: A Mathematical Invitation},
Cambridge University Press, Cambridge (2013).

\bibitem{squiral}
Baake M and Grimm U,
Squirals and beyond:\ Substitution tilings with singular 
continuous spectrum,
\textit{Ergodic Th.\ \& Dynam.\ Syst.} \textbf{34} (2014)
1077--1102; \texttt{arXiv:1205.1384}.

\bibitem{Haynes}
Baake M and Haynes A, 
A measure-theoretic result for approximation by Delone sets,
\textit{preprint}\newline \texttt{arXiv:1702.04839}.

\bibitem{BHL}
Baake M, Haynes A and Lenz D,
Averaging almost periodic functions along exponential
sequences, to appear in: \textit{Aperiodic Order. Vol.~2:
Crystallography and Almost Periodicity}, 
Baake M and Grimm U (eds.), 
Cambridge University Press, Cambridge (2017),
in press; 
\texttt{arXiv:1704.08120}.

\bibitem{BL}
Baake M and Lenz D,
Dynamical systems on translation bounded measures:\
Pure point dynamical and diffraction spectra, 
\textit{Ergodic Th.\ \& Dynam.\ Syst.} \textbf{24} (2004)
1867--1893; \texttt{arXiv:math.DS/0302231}.

\bibitem{BLvE}
Baake M, Lenz D and van Enter A C D,
Dynamical versus diffraction spectrum for structures with 
finite local complexity, 
\textit{Ergodic Th.\ \& Dynam.\ Syst.} \textbf{35} (2015) 2017--2043;
\texttt{arXiv:1307.7518}.

\bibitem{BP1}
Barreira L and Pesin Y,
\textit{Nonuniform Hyperbolicity},
Cambridge University Press, Cambridge (2007).

\bibitem{BP2}
Barreira L and Pesin Y,
\textit{Introduction to Smooth Ergodic Theory},
AMS, Providence, RI (2013).

\bibitem{Bartlett}
Bartlett A,
Spectral theory of $\ZZ^{d}$ substitutions,
\textit{Ergodic Th.\ \& Dynam.\ Syst.}, in press;
\texttt{arXiv:1410.8106}.

\bibitem{BerSol}
Berlinkov A and Solomyak B,
Singular substitutions of constant length,
\textit{preprint}
\texttt{arXiv:1705.00899}.

\bibitem{BF}
Berg C and Forst G,
\textit{Potential Theory on Locally Compact Abelian Groups},
Springer, Berlin (1975).

\bibitem{BS}
Bufetov A and Solomyak B,
On the modulus of continuity for spectral measures in 
substitution dynamics,
\textit{Adv.\ Math.} \textbf{260} (2014) 84--129;
\texttt{arXiv:1305.7373}.

\bibitem{Bugeaud}
Bugeaud Y,
\textit{Distribution Modulo One and Diophantine Approximation},
Cambridge University Press, Cambridge (2012).

\bibitem{CS}
Clark A and Sadun L,
When size matters:\ Subshifts and their related tiling spaces,
\textit{Ergodic Th.\ \& Dynam.\ Syst.} \textbf{23} (2003) 1043--57; 
\texttt{arXiv:math.DS/0201152}.

\bibitem{Cord}
Corduneanu C,
\textit{Almost Periodic Functions},
2nd English ed., Chelsea, New York (1989).

\bibitem{CFS}
Cornfeld I P, Fomin S V, Sinai Ya G,
\textit{Ergodic Theory}, Springer, New York (1982).

\bibitem{Dekking}
Dekking F M,
The spectrum of dynamical systems arising from substitutions 
of constant length,
\textit{Z.\ Wahrscheinlichkeitsth.\ verw.\ Geb.} \textbf{41} 
(1978) 221--239.

\bibitem{FSS}
Fan A-H, Saussol B and Schmeling J,
Products of non-stationary random matrices and multiperiodic
equations of several scaling factors,
\textit{Pacific J.\ Math.} \textbf{214} (2004) 31--54;
\texttt{arXiv:math.DS/0210347}.

\bibitem{Nat1}
Frank N P,
Multi-dimensional constant-length substitution sequences,
\textit{Topol.\ Appl.} \textbf{152} (2005) 44--69.

\bibitem{GLA}
Gil de Lamadrid J and Argabright L N,
Almost periodic measures,
\textit{Memoirs Amer.\ Math.\ Soc.} \textbf{85}, no.~428 
(1990).

\bibitem{Harman}
Harman G,
\textit{Metric Number Theory},
Oxford University Press, New York (1998).

\bibitem{Hart}
Hartinger J, Kainhofer R F and Tichy R F,
Quasi-Monte Carlo algorithms for unbounded,
weighted integration problems,
\textit{J.\ Complexity} \textbf{20} (2004) 654--668.

\bibitem{Hof}
Hof A,
On diffraction by aperiodic structures,
\textit{Commun.\ Math.\ Phys.} \textbf{169} (1995) 25--43. 

\bibitem{KHN}
Kamarul Haili H and Nair R,
The discrepancy of some real sequences,
\textit{Math.\ Scand.} \textbf{93} (2003) 268--274.

\bibitem{KN}
Kuipers L and Niederreiter H,
\textit{Uniform Distribution of Sequences},
reprint, Dover, New York (2006).

\bibitem{Lenz}
Lenz D,
Continuity of eigenfunctions of uniquely ergodic
dynamical systems and intensity of Bragg peaks,
\textit{Commun.\ Math.\ Phys.} \textbf{287}
(2009) 225--258;
\texttt{arXiv:math-ph/0608026}.

\bibitem{Neil}
Ma\~{n}ibo N,
Lyapunov exponents for binary substitutions of constant
length, \textit{preprint} \texttt{arXiv:1706.00451}.

\bibitem{MS}
Moody R V and Strungaru N,
Almost periodic measures and their Fourier transforms,
to appear in: \textit{Aperiodic Order. Vol.~2:
Crystallography and Almost Periodicity}, 
Baake M and Grimm U (eds.), 
Cambridge University Press, Cambridge (2017), 
in press.

\bibitem{Q}
Queff\'{e}lec M,
\textit{Substitution Dynamical Systems -- Spectral Analysis},
LNM 1294, 2nd ed., Springer, Berlin (2010).

\bibitem{RS}
Reed M and Simon B,
\textit{Methods of Modern Mathematical Physics I: Functional Analysis},
2nd ed., Academic Press, San Diego, CA (1980).

\bibitem{Rob}
Robinson E A, 
Symbolic dynamics and tilings of $\RR^{\! d}$,
\textit{Proc.\ Sympos.\ Appl.\ Math.} \textbf{60} (2004) 81--119.

\bibitem{Sobol}
Sobol I M,
Calculation of improper integrals using uniformly distributed sequences,
\textit{Soviet Math.\ Dokl.} \textbf{14} (1973) 734--738.

\bibitem{Sol}
Solomyak B,
Dynamics of self-similar tilings,
\textit{Ergod.\ Th.\ \& Dynam.\ Syst.}
\textbf{17} (1997) 695--738 and
\textbf{19} (1999) 1685 (Erratum).

\bibitem{Viana}
Viana M,
\textit{Lectures on Lyapunov Exponents},
Cambridge University Press, Cambridge (2013).

\bibitem{Vince}
Vince A,
Digit tiling of Euclidean space,
in:\ \textit{Directions in Mathematical Quasicrystals},
Baake M and Moody R V (eds.),
CRM Monograph Series vol.\ 13, AMS, Providence, RI (2000),
pp.\ 329--370.

\end{thebibliography}
\end{document}